\documentclass[hidelinks,onefignum,onetabnum]{siamart250211}

\usepackage[english]{babel}
\usepackage{color}
\usepackage{amsmath}
\usepackage{amsfonts}
\usepackage{amssymb}
\usepackage{graphicx}
\usepackage{algorithmic}
\usepackage{comment}
\usepackage[normalem]{ulem}
\usepackage{multirow}
\usepackage{booktabs}
\usepackage{indentfirst}
\usepackage{enumitem}
\usepackage{todonotes}

\headers{Preconditioned Truncated Single Sample Estimators}{T. Xu, D. Cai, H. Huang, E. Chow, Y. Xi} 

\title{Preconditioned Truncated Single-Sample Estimators for Scalable Stochastic Optimization}

\author{Tianshi Xu\thanks{Department of Mathematics, Emory University, Atlanta, 30322 (\email{tianshi.xu,yxi26@emory.edu})} The research of Y. Xi is supported by
NSF awards DMS-2338904.
\and Difeng Cai\thanks{Department of Mathematics, Southern Methodist University, Dallas, TX 75205 (\email{ddcai@smu.edu})}
\and Hua Huang\thanks{School of Computational Science and Engineering, Georgia Institute of Technology, Atlanta, GA 30332 (\email{huangh223,echow@cc.gatech.edu})} The research of H. Huang and E. Chow is
supported by NSF award OAC 2003683.
\and Edmond Chow\footnotemark[3]
\and Yuanzhe Xi\footnotemark[1]
}

\usepackage{amsopn}

\ifpdf\hypersetup{
  pdftitle={Bias-Reduced Preconditioned GP},
  pdfauthor={Authors}
}
\fi

\newcommand{\zbf}{\mathbf{z}}
\newcommand{\cdphiss}{\widetilde{\Phi}_{tss}}
\newcommand{\cdx}{\mathbf{x}} % CG solution
\newcommand{\cdq}{Q} % random iteration number
\newcommand{\imin}{i_{\min}}
\newcommand{\imax}{i_{\max}}
\newcommand{\eebb}{\mathbb{E}}
\newcommand{\vvbb}{\mathrm{Var}}
\newcommand{\Tjz}[1]{\mathbf{T}^{(\mathbf{z})}_{#1}}

\newcommand{\cdf}[1]{{#1}}

\begin{document}

\maketitle

\begin{abstract}
Many large-scale stochastic optimization algorithms involve repeated solutions of linear systems or evaluations of log-determinants. In these algorithms, computing exact solutions is often unnecessary; it is often computationally more efficient to construct unbiased stochastic estimators with controlled variance. However, classical iterative solvers incur truncation bias, whereas unbiased Krylov-based estimators typically exhibit high variance and potential numerical instability. To mitigate these issues, we introduce the Preconditioned Truncated Single-Sample (PTSS) estimators---a family of stochastic Krylov methods that integrate preconditioning with truncated Lanczos iterations. PTSS yields low-variance, stable estimators for linear system solutions, log-determinants, and their derivatives. We establish theoretical results for their mean, variance, and concentration properties, explicitly quantifying the variance reduction induced by preconditioning. Numerical experiments confirm that PTSS achieves superior stability and variance control compared with existing unbiased and biased alternatives, providing an efficient and practical framework for stochastic optimization.
\end{abstract}

% REQUIRED
\begin{keywords} 
Preconditioning, Krylov subspace methods, stochastic optimization, trace estimation, Gaussian processes
\end{keywords}

\begin{AMS}
   	60G15, 65F08, 65F10, 68W25
\end{AMS}

\section{Introduction}\label{sec:intro}
Large-scale stochastic optimization problems are central to diverse applications, 
including variational inference, Bayesian experimental design, probabilistic machine learning, 
and hyperparameter selection~\cite{Blei03042017,Murphy_2023,scott2025designing,hyperop}. 
In many of these settings, the primary computational bottleneck is the repeated solution of large linear systems 
or the evaluation of log-determinants, which often dominate the overall cost of objective 
and gradient computations.

A motivating example is Gaussian Process (GP) modeling, a non-parametric Bayesian framework widely used for regression and classification in science and engineering applications~\cite{Rasmussen_Williams_2005,Murphy_2022,Murphy_2023,Huang2026}. Given any training data $\mathbf{X}\in\mathbb{R}^{n_1\times d}$, noisy training observation $\mathbf{y}\in\mathbb{R}^{n_1}$, and testing data $\mathbf{X}_\ast\in\mathbb{R}^{n_2\times d}$, a standard GP model assumes that the noise-free testing observation $\mathbf{y}_\ast\in\mathbb{R}^{n_2}$ follows the joint distribution \begin{equation} \begin{bmatrix} \mathbf{y} \\ \mathbf{y}_\ast \end{bmatrix} \sim \mathcal{N} \begin{pmatrix} \mathbf{0}, f^2 \begin{bmatrix} \kappa(\mathbf{X},\mathbf{X})+\mu\mathbf{I} & \kappa(\mathbf{X},\mathbf{X}_{\ast})\\ \kappa(\mathbf{X}_{\ast},\mathbf{X}) & \kappa(\mathbf{X}_{\ast},\mathbf{X}_{\ast}) \end{bmatrix} \end{pmatrix}. \end{equation} 
Here $f$ and $\mu$ are scalar hyperparameters, $\mathbf{I}$ is the identity matrix, $\kappa:\mathbb{R}^d\times\mathbb{R}^d\to\mathbb{R}$ is a kernel function, and $\kappa(\mathbf{U},\mathbf{V})$ denotes a kernel matrix with entries
\[
[\kappa(\mathbf{U},\mathbf{V})]_{ij} \;=\; \kappa(\mathbf{u}_i,\mathbf{v}_j),
\]
where $\mathbf{u}_i$, $\mathbf{v}_j$ denote $i$-th and $j$-th rows of $\mathbf{U}$ and $\mathbf{V}$, respectively. Common kernel functions include the Gaussian (RBF), Mat\'ern, and exponential kernels, each depending on one or more hyperparameters. For instance, the RBF kernel$$\kappa(\mathbf{x},\mathbf{y})=\exp(-||\mathbf{x}-\mathbf{y}||_2^2/2l^2)$$ depends on a length-scale parameter $l$.

%allison2023leveraging
The quality of a GP model depends on both the choice of the kernel~\cite{duvenaud2013structure} and its associated hyperparameters~\cite{pmlr-v139-potapczynski21a, pmlr-v162-wenger22a}. Assuming the kernel function is fixed, the remaining task is to estimate hyperparameters by minimizing the negative log marginal likelihood (NLML):
\begin{equation}\label{eq:loss}
L(\Theta) \;=\; \frac{1}{2}\left(\mathbf{y}^{\top}\widehat{\mathbf{K}}^{-1}\mathbf{y} 
+ \log|\widehat{\mathbf{K}}| + n_1\log 2\pi\right),
\end{equation}
where $\widehat{\mathbf{K}} = K(\mathbf{X},\mathbf{X}) + \mu\mathbf{I}$ and $\Theta = (\mu, f, l)$ denotes the hyperparameter set. The derivatives of NLML are given by
\begin{equation}\label{eq:loss_grad}
\frac{\partial L}{\partial \theta}
\;=\;
-\frac{1}{2}\!\left(
\mathbf{y}^{\top}\widehat{\mathbf{K}}^{-1}
\frac{\partial \widehat{\mathbf{K}}}{\partial \theta}
\widehat{\mathbf{K}}^{-1}\mathbf{y}
\;-\;
\mathrm{tr}\!\Big(\widehat{\mathbf{K}}^{-1}
\frac{\partial \widehat{\mathbf{K}}}{\partial \theta}\Big)
\right),
\qquad \theta\in\Theta.
\end{equation}

When computed exactly---for example, using a Cholesky factorization---both the NLML and its gradients require $\mathcal{O}(n_1^3)$ complexity. 
This cubic scaling quickly renders direct GP training impractical for large datasets. 
To mitigate this challenge, two main families of approximation strategies have been developed: 
(i) simplifying the GP model itself~\cite{pmlr-v5-titsias09a, hensman2013gaussianprocessesbigdata,wilson2015thoughtsmassivelyscalablegaussian, pmlr-v80-pleiss18a,covariance,ambikasaran2015fast,minden2017fast,geoga2020scalable},
and (ii) approximating the NLML objective~\eqref{eq:loss} and its gradients~\eqref{eq:loss_grad} through iterative methods \cite{Gardner_Pleiss_Weinberger_Bindel_Wilson_2018,pmlr-v162-wenger22a,wagner2025preconditionedadditivegaussianprocesses}. 
In the latter category, the term $\widehat{\mathbf{K}}^{-1}\mathbf{y}$ is typically estimated using Krylov subspace methods~\cite{Saad_2003}, 
while trace terms are approximated with stochastic estimators such as Hutchinson’s method~\cite{Hutchinson_1989, doi:10.1137/1.9781611976496.16} 
or stochastic Lanczos quadrature~\cite{Ubaru_Chen_Saad_2017}. 

In this paper, we focus on constructing stochastic Krylov estimators whose gradient estimates are unbiased (or bias-reduced) at a prescribed computational budget, rather than insisting on fully converged inner solves. Instead of computing~\eqref{eq:loss_grad} or related quantities exactly, we replace the costly matrix operations with low-variance stochastic estimators. This enables the use of a stochastic gradient-based outer loop while greatly reducing the computational expense of each iteration. Building on this property, we develop a class of \emph{Preconditioned Truncated Single-Sample (PTSS)} estimators, where truncation refers to a fixed number of iterations, for scalable estimation of linear system solutions and log-determinants across a wide range of stochastic optimization problems. By combining preconditioning with a stochastic truncated Lanczos approximation, PTSS produces low-variance estimators without incurring the expense of unnecessarily accurate solutions. Compared to deterministic truncation with the same expected computational cost, this randomized approach can significantly mitigate bias, while the estimator variance is controlled given a sufficiently accurate preconditioner.

{Throughout, we measure computational cost primarily in terms of matrix--vector products (MVPs) with the kernel matrix, since MVPs dominate wall-clock time in large-scale GP training.
Each Krylov iteration typically requires one MVP and one application of the preconditioner.
Thus a deterministic $m$-step truncated method costs $\approx m$ MVPs.
We use this MVP accounting to compare methods under matched computational budgets.}

The main contribution of this paper is a unified theoretical framework for TSS Krylov estimators, accompanied by a rigorous analysis of their bias, variance, and concentration properties. We show that randomized truncation of Krylov iterations yields mean-equivalent estimates to deterministic truncation. Furthermore, we derive explicit variance bounds in terms of the matrix condition number, sampling distribution, and truncation window. Two concrete variants are introduced: TSS--Solve for linear system solutions and TSS--LogQF for log-determinant estimation. The analysis provides a clear spectral interpretation linking preconditioning to both variance reduction and numerical stability, and identifies a specific sampling distribution that minimizes the variance bound as a function of the condition number. Together, these results establish a principled foundation for variance-controlled stochastic estimation in large-scale optimization.

This manuscript is organized as follows.
Section~\ref{sec: related works} reviews relevant background on conjugate gradient methods, preconditioning strategies for stochastic trace estimation, and unbiased Krylov estimators.
Section~\ref{sec:tss} introduces the TSS estimators and establishes variance and concentration bounds.
Section~\ref{subsec:stable-precond-lanczos} describes a robust and efficient implementation of TSS estimators using a unified preconditioned Lanczos framework.
Section~\ref{sec: numerical experiments} presents numerical results on both synthetic problems and real-world applications, demonstrating the effectiveness of the proposed estimators.
Finally, Section~\ref{sec: conclusion} concludes the paper with a summary of findings and future directions.

\section{Background}\label{sec: related works}

In this section, we review the application of block conjugate gradient (CG) methods in stochastic trace and derivative estimation, and then introduce the corresponding unbiased estimators.

\subsection{Block CG for multiple right-hand sides}

We first review block CG methods in the context of three fundamental computational tasks:  
(i) solving sequences of linear systems with a common symmetric positive definite (SPD) matrix \( \mathbf{A}\) and multiple (often random) right-hand sides,  
(ii) using the stochastic Lanczos algorithm to approximate the log-determinant of \( \mathbf{A}\), and  
(iii) approximating derivatives of inverse-quadratic and log-determinant terms when the matrix depends on parameters, \( \mathbf{A}=\mathbf{A}(\theta)\).  
These operations underlie objective and gradient evaluations in many large-scale stochastic optimization problems, including those in \eqref{eq:loss}--\eqref{eq:loss_grad}.

More specifically, let \(\mathbf{A}\in\mathbb{R}^{n\times n}\) be SPD and \(\mathbf{y}\in\mathbb{R}^n\) a vector. 
The quantities of interest can be written as
\begin{align}
    \mathcal{Q}(\mathbf{y}) &:= \mathbf{y}^\top \mathbf{A}^{-1}\mathbf{y}, \label{eq:quad_form}\\
    d\mathcal{Q}(\theta) &:= \mathbf{y}^\top \mathbf{A}^{-1}\tfrac{\partial \mathbf{A}}{\partial \theta}\mathbf{A}^{-1}\mathbf{y}, \label{eq:quad_form_grad}\\
    \mathcal{T}(\mathbf{A}) &:= \log|\mathbf{A}| = \mathrm{tr}(\log \mathbf{A}), \label{eq:logdet}\\
    d\mathcal{T}(\theta) &:= \mathrm{tr}\!\left(\mathbf{A}^{-1}\tfrac{\partial \mathbf{A}}{\partial \theta}\right). \label{eq:logdet_grad}
\end{align}
Evaluating or approximating these quantities efficiently and reliably is the focus of the block CG methods discussed below.

To approximate $\mathcal{Q}(\mathbf{y})$, we apply $m$ iterations of CG with initial guess $\mathbf{x}_{0}=0$ to the system
\begin{equation}\label{eq:cg_Ax}
    \mathbf{A}\mathbf{x} = \mathbf{y},
\end{equation}
yielding an approximate solution $\mathbf{x}_{m}$. 
Then
\begin{equation}\label{eq:quad_approx}
    \mathcal{Q}(\mathbf{y})\approx\widehat{\mathcal{Q}}(\mathbf{y}) = \mathbf{y}^\top \mathbf{x}_{m}, 
    \quad 
    d\mathcal{Q}(\theta)\approx \widehat{d\mathcal{Q}}(\theta) = (\mathbf{x}_{m})^\top \frac{\partial \mathbf{A}}{\partial \theta}\,\mathbf{x}_{m}.
\end{equation}

For the log-determinant, we employ stochastic Lanczos quadrature (SLQ) \cite{Ubaru_Chen_Saad_2017}. 
Sampling $k_z$ test vectors $\mathbf{z}^i\sim\mathcal{N}(0,I)$, we solve
\begin{equation}\label{eq:cg_Au}
    \mathbf{A}\mathbf{u}^i = \mathbf{z}^i, 
    \quad i=1,\dots,k_z,
\end{equation}
using block CG. 
Assume all linear systems in \eqref{eq:cg_Au} are solved using the same number of CG iterations. Then each run generates both approximate solutions $\mathbf{u}^i_{m}$ and a tridiagonal matrix $\mathbf{T}_{m}^{(\mathbf{z}^i)}\in\mathbb{R}^{m\times m}$ encoding spectral information (See Appendix~\ref{sec: cg eig est}). 
The log-determinant is then approximated by
\begin{equation}\label{eq:logdet_approx}
    \widehat{\mathcal{T}}(\mathbf{A}) 
    = \frac{1}{k_z}\sum_{i=1}^{k_z}\|\mathbf{z}^i\|^2 \,\mathbf{e}_1^\top \log(\mathbf{T}_{m}^{(\mathbf{z}^i)})\,\mathbf{e}_1,
\end{equation}
where $\mathbf{e}_1$ denotes the first column of the identity matrix whose size is the same as the tridiagonal matrix $\mathbf{T}_{m}^{(\mathbf{z}^i)}$.

Trace derivatives such as $d\mathcal{T}(\theta)$ can be estimated with Hutchinson’s method \cite{Hutchinson_1989}:
\begin{equation}\label{eq:trace_approx}
    \widehat{d\mathcal{T}}(\theta) 
    = \frac{1}{k_z}\sum_{i=1}^{k_z} (\mathbf{u}^i_{m})^\top \frac{\partial \mathbf{A}}{\partial \theta}\,\mathbf{z}^i,
\end{equation}
reusing the approximate solutions $\mathbf{u}^i_{m}$ from \eqref{eq:cg_Au}.

Since all solves in \eqref{eq:cg_Ax} and \eqref{eq:cg_Au} share the same coefficient matrix $\mathbf{A}$, they can be performed simultaneously using block CG. 
This allows us to compute $\mathbf{x}_{m}$ and $\{\mathbf{u}^i_{m}\}_{i=1}^{k_z}$ in parallel, leveraging block matrix-vector multiplications to amortize costs across multiple right-hand sides. 
A generic procedure is summarized in Algorithm~\ref{alg:BCG}.

\begin{algorithm}[htbp]
\caption{Block CG for estimating $\mathcal{Q}(\mathbf{y})$, $d\mathcal{Q}(\theta)$, $\mathcal{T}(\mathbf{A})$, and $d\mathcal{T}(\theta)$\label{alg:BCG}}
\begin{algorithmic}[1]
\STATE{\bf Input:} SPD matrix $\mathbf{A}$, vector $\mathbf{y}$, number of test vectors $k_z$, maximum iterations $m$.
\STATE{\bf Output:} Approximations of $\mathcal{Q}(\mathbf{y})$, $d\mathcal{Q}(\theta)$, $\mathcal{T}(\mathbf{A})$, and $d\mathcal{T}(\theta)$.
\STATE Sample $k_z$ random vectors $\mathbf{z}^i\sim\mathcal{N}(0,I)$.
\STATE Run $m$ steps of block CG on $\mathbf{A}[\mathbf{x},\mathbf{u}^1,\dots,\mathbf{u}^{k_z}] = [\mathbf{y},\mathbf{z}^1,\dots,\mathbf{z}^{k_z}]$.
\STATE Form approximations $\widehat{\mathcal{Q}}(\mathbf{y}), \widehat{d\mathcal{Q}}(\theta)$ using \eqref{eq:quad_approx}.
\STATE Form approximations $\widehat{\mathcal{T}}(\mathbf{A}), \widehat{d\mathcal{T}}(\theta)$ using \eqref{eq:logdet_approx} and \eqref{eq:trace_approx}.
\STATE Return all estimates.
\end{algorithmic}
\end{algorithm}

This block formulation underlies many popular implementations in existing software packages due to its theoretical appeal: all required estimates---such as inverse quadratic forms, log-determinants, and their derivatives---can be computed from matrix-vector products with \(\mathbf{A}\) and its derivatives \cite{Gardner_Pleiss_Weinberger_Bindel_Wilson_2018}. However, despite its efficiency and elegance in theory, the approach is often numerically unstable in practice due to the loss of orthogonality in the Krylov subspace basis, especially for ill-conditioned problems or long iteration runs.

\subsection{Preconditioning}
Preconditioning aims to improve the spectral properties of the coefficient matrix, thereby accelerating the convergence of Krylov subspace methods~\cite{Saad_2003}. In stochastic settings, it serves an analogous purpose by enhancing computational efficiency through variance reduction in estimators for traces and log-determinants, including $\widehat{\mathcal{T}}(\mathbf{A})$ and $d\widehat{\mathcal{T}}(\theta)$~\cite{pmlr-v162-wenger22a, Gardner_Pleiss_Weinberger_Bindel_Wilson_2018, Zhao_Xu_Huang_Chow_Xi_2024}.
Crucially, while eigenvalue clustering is often sufficient for accelerating linear solvers, effective variance reduction in our context imposes a stronger condition: $\mathbf{M}$ must spectrally approximate $\mathbf A$, i.e., $\mathbf M^{-1}\mathbf A\approx\mathbf I$.

We begin with the classical definition of preconditioning for solving linear systems. Consider a linear system $\mathbf{A}\mathbf{x}=\mathbf{y}$. 
Introducing a left preconditioner $\mathbf{M}$ yields the equivalent system
\begin{equation}\label{eq:pcg_Ax}
    \mathbf{M}^{-1}\mathbf{A}\mathbf{x} = \mathbf{M}^{-1}\mathbf{y},
\end{equation}
where $\mathbf{M}$ is chosen as an easily invertible approximation of $\mathbf{A}$.  
Although both systems share the same solution, 
the preconditioned matrix $\mathbf{M}^{-1}\mathbf{A}$ typically has a more favorable spectrum.  
The convergence of the preconditioned conjugate gradient (PCG) method satisfies
\[
    \|\mathbf{err}_{m}\|_{\mathbf{A}}
    \;\leq\; 2\|\mathbf{err}_{0}\|_{\mathbf{A}}
    \left(\frac{\sqrt{\kappa}-1}{\sqrt{\kappa}+1}\right)^{m},
\]
where $\kappa$ denotes the condition number of $\mathbf{M}^{-1}\mathbf{A}$ 
and $\mathbf{err}_{m}$ is the error after $m$ iterations.  
An effective preconditioner can therefore greatly reduce the computational cost of evaluating 
$\mathcal{Q}(\mathbf{y})$ and $d\mathcal{Q}(\theta)$.

Preconditioning further benefits stochastic approximations of log-determinants and trace derivatives.  
For example,
\begin{equation}\label{eq:precond_logdet}
    \log|\mathbf{A}|
    = \log|\mathbf{M}| + \log|\mathbf{M}^{-1/2}\mathbf{A}\mathbf{M}^{-1/2}|,
\end{equation}
so that stochastic Lanczos quadrature estimators can be applied to
$\mathbf{M}^{-1/2}\mathbf{A}\mathbf{M}^{-1/2}$, yielding reduced variance.

Similarly, for the trace derivative,
\begin{equation}\label{eq:precond_trace}
    \mathrm{tr}\!\left(\mathbf{A}^{-1}\frac{\partial \mathbf{A}}{\partial \theta}\right)
    = \mathrm{tr}\!\left(\mathbf{M}^{-1}\frac{\partial \mathbf{M}}{\partial \theta}\right)
    + \mathrm{tr}\!\left(\mathbf{A}^{-1}\frac{\partial \mathbf{A}}{\partial \theta}
      - \mathbf{M}^{-1}\frac{\partial \mathbf{M}}{\partial \theta}\right),
\end{equation}
where the first term can often be computed exactly at low cost, 
and the second term typically exhibits smaller variance 
due to the effect of preconditioning.

\subsection{Unbiased Krylov methods}
A well-known limitation of truncated Krylov subspace methods (e.g., CG or Lanczos stopped after $m$ iterations) is that they yield \emph{biased} approximations of scalar quantities of interest. Potapczynski et al.~\cite{pmlr-v139-potapczynski21a} demonstrated that truncated CG systematically underestimates inverse quadratic forms and overestimates log-determinants.
In practice, these biases can be especially problematic when Krylov approximations are embedded in stochastic optimization, since persistent bias in gradient or objective estimates violates the unbiasedness assumptions underlying many convergence guarantees~\cite{bach2024learning}. These challenges have led to the development and analysis of unbiased estimators based on Krylov subspace methods~\cite{pmlr-v139-potapczynski21a}.

The starting point for constructing unbiased estimators is to express a target quantity $\Phi$ as a finite sum:
\begin{equation}
    \Phi = \sum_{i=1}^n \Delta_i.
    \label{eq:Phi summation}
\end{equation}
For instance, when solving a linear system $\mathbf{A}\mathbf{x}=\mathbf{y}$ with CG, let $\mathbf{x}_i$ denote the $i$-th iterate (with $\mathbf{x}_0=0$). Then
\begin{equation*}
    \mathbf{x} = \mathbf{x}_n = \sum_{i=1}^n \underbrace{(\mathbf{x}_i - \mathbf{x}_{i-1})}_{\Delta_i}.
\end{equation*}
A natural but biased estimator is obtained by truncating the sum at $m$ when $m<n$:
\begin{equation}
    \widetilde{\Phi} = \sum_{i=1}^{m} \Delta_i = \mathbf{x}_{m}.
    \label{eq:Phi summation biased}
\end{equation}
Since $\widetilde{\Phi}\neq \Phi$ in general, such truncations are biased in the sense of~\cite{pmlr-v139-potapczynski21a}.  

To remove this bias, the deterministic cutoff $m$ is replaced by a random truncation index $\cdq \sim p_{\cdq}$ supported on $\{1,2,\ldots,n\}$. 
Here, $p_{\cdq}$ denotes the probability mass function; examples include geometric distributions where $\mathbb{P}(\cdq = j)\propto e^{-0.5j}$ or $\mathbb{P}(\cdq = j)\propto 2^{-j}$.
Two widely used constructions are the \emph{single-sample (SS)} estimator~\cite{10.1214/15-STS523} and the \emph{Russian roulette (RR)} estimator~\cite{Kahn_1955}:  
\begin{align}
    [\Phi]_{ss}\big|_{\cdq=j} &= \frac{\Delta_j}{\mathbb{P}(\cdq=j)}, 
    \label{eq:ss estimator}\\
    [\Phi]_{rr}\big|_{\cdq=j} &= \sum_{i=1}^{j}\frac{\Delta_i}{\mathbb{P}(\cdq\geq i)}.
    \label{eq:rr estimator}
\end{align}
Both satisfy $\mathbb{E}([\Phi]_{ss}) = \mathbb{E}([\Phi]_{rr}) = \Phi$. The SS estimator is simple and inexpensive to compute, while the RR estimator achieves lower variance at the cost of additional computation.

We now describe how to construct unbiased Krylov estimators for the four quantities in \eqref{eq:quad_form}--\eqref{eq:logdet_grad}.

To estimate the inverse quadratic form $\mathcal{Q}(\mathbf{y}) = \mathbf{y}^\top \mathbf{A}^{-1} \mathbf{y}$, 
we solve $\mathbf{A} \mathbf{x} = \mathbf{y}$ using CG and apply 
SS or RR estimators to the CG iterates.
Let $\mathbf{x}_0 = 0$ and $\mathbf{x}_i$ denote the $i$-th iterate of CG. Given a random truncation index $j \sim p_{\cdq}$, define:
\[
[\mathbf{x}]_{\mathrm{ss}} \big|_{\cdq=j} = \frac{\mathbf{x}_j - \mathbf{x}_{j-1}}{\mathbb{P}(\cdq = j)},
\qquad
[\mathbf{x}]_{\mathrm{rr}} \big|_{\cdq=j} = \sum_{i=1}^j \frac{\mathbf{x}_i - \mathbf{x}_{i-1}}{\mathbb{P}(\cdq \ge i)}.
\]
Then the unbiased inverse quadratic form estimators are:
\[
\widetilde{\mathcal{Q}}_{\mathrm{ss}}(\mathbf{y}) = \mathbf{y}^\top [\mathbf{x}]_{\mathrm{ss}},
\quad
\widetilde{\mathcal{Q}}_{\mathrm{rr}}(\mathbf{y}) = \mathbf{y}^\top [\mathbf{x}]_{\mathrm{rr}},
\quad
\mathbb{E}\, \widetilde{\mathcal{Q}}_{\mathrm{ss}/\mathrm{rr}}(\mathbf{y}) = \mathcal{Q}(\mathbf{y}).
\]

For the gradient $d\mathcal{Q}(\theta) = \mathbf{y}^\top \mathbf{A}^{-1} \tfrac{\partial \mathbf{A}}{\partial\theta} \mathbf{A}^{-1} \mathbf{y}$, we use the estimator
\[
\widetilde{d\mathcal{Q}}_{\mathrm{ss}/\mathrm{rr}}(\theta)= [\mathbf{x}^{\top}\tfrac{\partial \mathbf{A}}{\partial\theta}\mathbf{x}]_{\mathrm{ss}/\mathrm{rr}},
\quad
\mathbb{E}\, \widetilde{d\mathcal{Q}}_{\mathrm{ss}/\mathrm{rr}}(\theta) = d\mathcal{Q}(\theta).
\]

To estimate $d\mathcal{T}(\theta) = \mathrm{tr}\!\left(\mathbf{A}^{-1} \tfrac{\partial \mathbf{A}}{\partial\theta} \right)$,
we use Hutchinson’s identity:
\[
d\mathcal{T}(\theta) = \mathbb{E}_{\mathbf{z}} \left[ \mathbf{z}^\top \mathbf{A}^{-1} \tfrac{\partial \mathbf{A}}{\partial\theta} \mathbf{z} \right].
\]
Let $\{\mathbf{z}^i\}_{i=1}^k$ be $k$ standard Gaussian probes.
For each $\mathbf{z}^i$, construct an unbiased estimator $[\mathbf{u}^i]$ for $\mathbf{A}^{-1} \mathbf{z}^i$ as above.
Then the gradient estimators are given by:
\[
\widetilde{d\mathcal{T}}_{\mathrm{ss}/\mathrm{rr}}(\theta)
= \frac{1}{k} \sum_{i=1}^k [\mathbf{u}^i]_{\mathrm{ss}/\mathrm{rr}}^\top \tfrac{\partial \mathbf{A}}{\partial\theta} \, \mathbf{z}^i,
\qquad
% \widehat{d\mathcal{T}}^{\mathrm{rr}}(\theta)
% = \frac{1}{k} \sum_{i=1}^k [\mathbf{x}_i]_{\mathrm{rr}}^\top \tfrac{\partial \mathbf{A}}{\partial\theta} \, \mathbf{z}_i,
% \quad
\mathbb{E}\, \widetilde{d\mathcal{T}}_{\mathrm{ss}/\mathrm{rr}}(\theta) = d\mathcal{T}(\theta).
\]

To estimate the log-determinant $\mathcal{T}(\mathbf{A}) = \mathrm{tr}(\log \mathbf{A})$, one approach is to employ stochastic Lanczos quadrature implemented via the block CG framework. 
For each probe vector $\mathbf{z} \sim \mathcal{N}(0, I)$,  $j$ steps of CG are performed to solve $\mathbf{A} \mathbf{u} = \mathbf{z}$,  where the truncation index $j \sim p_{\cdq}$ is drawn from a prescribed distribution. 
The CG recurrence scalars $\alpha,\beta$ collected during these steps are then used to reconstruct the tridiagonal matrix $\Tjz{j}$ that would have been generated by $j$ steps of the Lanczos algorithm initialized with $\mathbf{q}_1 = \mathbf{z} / \|\mathbf{z}\|$ in exact arithmetic. See Appendix \ref{sec: cg eig est} for more details of this relation.
Let:
\[
s_j(\mathbf{z}) = \mathbf{e}_1^\top \log\big( \mathbf{T}^{(\mathbf{z})}_j \big) \mathbf{e}_1,
\qquad
\Delta_j(\mathbf{z}) = s_j(\mathbf{z}) - s_{j-1}(\mathbf{z}),\qquad s_0(\mathbf{z})=0,
\]
where the length of the standard basis vector $\mathbf{e}_1=[1,0,\dots,0]^T$ is equal to $j$ in $s_j$. Here we use the notation $\mathbf{e}_1$ (independent of $j$) simply for notational convenience.
Applying the SS or RR estimator based on the scaled sequence $\{ ||\mathbf{z}||^2\Delta_j(\mathbf{z}) \}$ produces unbiased estimates of $\mathbf{z}^\top \log(\mathbf{A}) \mathbf{z}$.
Averaging over $k$ probes gives the estimator:
\[
\widetilde{\mathcal{T}}_{\mathrm{ss}/\mathrm{rr}}(\mathbf{A})
= \frac{1}{k} \sum_{i=1}^k \left[ (\mathbf{z}^i)^\top \log(\mathbf{A}) \, \mathbf{z}^i \right]_{\mathrm{ss}/\mathrm{rr}},
\qquad
\mathbb{E}\, \widetilde{\mathcal{T}}_{\mathrm{ss}/\mathrm{rr}}(\mathbf{A}) = \mathcal{T}(\mathbf{A}).
\]

In practice, SS and RR estimators involve distinct tradeoffs:
\begin{itemize}[leftmargin=0.3in]
    \item \textbf{RR estimators} achieve lower variance by aggregating contributions from all iterates up to a random truncation index. However, this aggregation incurs significantly higher computational cost and memory usage, especially when the sampled index is large.
    \item \textbf{SS estimators} are computationally efficient, requiring only a single increment. Yet this simplicity comes at the cost of higher estimator variance.
\end{itemize}

To balance these tradeoffs in large-scale applications, it is essential to understand and control the variance behavior of SS-type estimators.

\section{Truncated Single-Sample (TSS) Krylov estimators}
\label{sec:tss}
In this section, we propose a class of  \emph{truncated single-sample} (TSS) Krylov estimators for two building blocks used in \eqref{eq:quad_form}--\eqref{eq:logdet_grad}: (i) solving linear systems based on CG and (ii) estimating the inverse quadratic form of $\log(\mathbf{A})$. Different from the standard SS estimator in \eqref{eq:ss estimator} where the random truncation $Q$ can range from $1$ to $n$, the $Q$ in TSS is sampled from a potentially smaller range $[\imin,\imax]$ customized by the user, namely,
\[
  \cdq \sim p_{\cdq} \quad \text{supported on} \quad \{i_{\min}, i_{\min}\!+\!1, \ldots, i_{\max}\},
\]
with $1 \le i_{\min} < i_{\max} \le n$. The lower bound $i_{\min}$ helps mitigate the variance introduced by very few iterations, while the upper bound $i_{\max}$ balances computational cost, accuracy, and variance. The SS estimator in \eqref{eq:ss estimator} corresponds to the choice $(i_{\min}, i_{\max}) = (1,n)$. We next introduce the definition and general properties of the TSS estimator.

\noindent\textbf{Notation.} All vector norms $\|\cdot\|$ are Euclidean. For a vector-valued random variable $X$, the variance is denoted by
$\mathrm{Var}(X) := \mathbb{E}\,\left(\|X-\mathbb{E}X\|^2\right).$

With this notation in place, and given the summation representation
\begin{equation}
\label{eq:Phi-sum-again}
\Phi \;=\; \sum_{i=1}^{n} \Delta_i.
\end{equation}
We now define the \emph{truncated single-sample (TSS)} estimator:
\begin{equation}
\label{eq:TSS-estimator-again}
\widetilde{\Phi}_{\mathrm{tss}} \;=\; \frac{\Delta_{\cdq}}{\mathbb{P}(\cdq)} + \sum_{i=1}^{i_{\min}-1} \Delta_i\, 
\qquad \cdq \sim p_{\cdq}\ \text{on}\ \{i_{\min},\ldots,i_{\max}\}.
\end{equation}

Having introduced the TSS estimator, we now establish its basic properties. In particular, we show that its expectation coincides with the deterministic truncation at $i_{\max}$ and represent its variance in terms of $\Delta_i$. 

\begin{proposition}
\label{prop:TSSbasic}
Let $\cdphiss$ be the TSS estimator in \eqref{eq:TSS-estimator-again}. Then
\begin{equation}
\label{eq:TSS-expectation-again}
\mathbb{E}\big[\cdphiss\big] \;=\; \sum_{i=1}^{i_{\max}} \Delta_i\;\;\text{and}\;\;
\mathrm{Var}\!\left(\widetilde{\Phi}_{\mathrm{tss}}\right)= \left(\sum_{j=\imin}^{\imax} \frac{||\Delta_j||^2}{\mathbb{P}(\cdq=j)}\right) - ||\Delta_*||^2,
\end{equation}
where $\Delta_*:=\sum\limits_{j=\imin}^{\imax}\Delta_j.$
As a result, for $\Phi$ defined in \eqref{eq:Phi-sum-again}, $\cdphiss$ is unbiased if $\imax=n$.
\end{proposition}
\begin{proof}
    Computing from the definition gives
    \[
    \mathbb{E}\big[\widetilde{\Phi}_{\mathrm{tss}}\big] = \sum_{j=\imin}^{\imax} \mathbb{P}(Q=j) \cdphiss|_{Q=j}
    = \sum_{j=\imin}^{\imax} \Delta_j + \sum_{i=1}^{i_{\min}-1} \Delta_i = \sum_{i=1}^{\imax} \Delta_i,
    \]
    where we have used the fact that $\sum\limits_{j=\imin}^{\imax} \mathbb{P}(Q=j) =1.$
    For $\Phi$ in \eqref{eq:Phi-sum-again} and $\imax=n$, it follows immediately that $\cdphiss$ is unbiased, i.e.
    $\mathbb{E}\big[\cdphiss\big]=\sum_{i=1}^{n} \Delta_i=\Phi.$
    This proves the claims about $\mathbb{E}\big[\cdphiss\big]$.
    
    Next we derive the variance estimate.
    For $\cdq=j$, we have
\begin{equation*}
    \cdphiss|_{\cdq=j}-\eebb\big[\cdphiss\big]= \frac{\Delta_j}{\mathbb{P}(\cdq=j)}-\sum_{i=\imin}^{\imax} \Delta_i = \frac{\Delta_j}{\mathbb{P}(\cdq=j)}-\Delta_*.
\end{equation*}
    Direct computation yields
\begin{equation*}
\begin{aligned}
\mathrm{Var}\!\left(\cdphiss\right)
&=\sum_{j=i_{\min}}^{i_{\max}}\mathbb{P}(\cdq=j)\left\|\frac{\Delta_j}{\mathbb{P}(\cdq=j)}-\Delta_*\right\|^2\\
    &= \sum_{j=\imin}^{\imax} \frac{||\Delta_j||^2}{\mathbb{P}(\cdq=j)} -\sum_{j=\imin}^{\imax} 2\Delta_j\cdot\Delta_* + \sum_{j=\imin}^{\imax} \mathbb{P}(\cdq=j)||\Delta_*||^2 \\
    &= \left( \sum_{j=\imin}^{\imax} \frac{||\Delta_j||^2}{\mathbb{P}(\cdq=j)}\right)  - 2||\Delta_*||^2+||\Delta_*||^2,
    % &= \left(\sum_{j=\imin}^{\imax} \frac{||\Delta_j||^2}{\mathbb{P}(\cdq=j)}\right)  - ||\Delta_*||^2
    % & \leq \sum_{j=\imin}^{\imax} \frac{||\Delta_j||^2}{\mathbb{P}(\cdq=j)},
\end{aligned}
\end{equation*}
where we have used the definition of $\Delta_*$ and 
$\sum\limits_{j=\imin}^{\imax}\mathbb{P}(\cdq=j)=1.$
The proof is complete.
\end{proof}
Note that even though the estimator $\cdphiss$ is generally biased for the case $\imax<n$, the identity in \eqref{eq:TSS-expectation-again} shows that the expectation coincides with the $\imax$-th iterate (assuming the $0$-th iterate is defined as $0$).
Thus the estimator in expectation matches the deterministic truncation at $i_{\max}$, often at a much lower cost due to randomized early stopping.  The variance expression shows that the choice of sampling distribution $\mathbb{P}(\cdq=j)$ is also important.  

In the following  subsections, we examine two specific instances, TSS-Solve and TSS-LogQF. These two estimators illustrate the concrete choice of $\Delta_j$ in \eqref{eq:Phi-sum-again} for each problem and motivate suitable choices for the sampling distribution.

\subsection{TSS-Solve estimator}
\label{sub:TSS-CG}
We begin by applying the TSS estimator to the solution of a linear system. In this context, the terms $\Delta_j$ correspond to the differences between successive CG iterates. The following result establishes bounds on the expectation and variance of the resulting TSS-Solve estimator, with explicit dependence on the condition number of $\mathbf{A}$, the sampling range and the sampling distribution $p_{\cdq}$.
A more refined analysis of the variance bound and a particular choice $p_{\cdq}$ is presented in Section \ref{sub:discussion}.

\begin{theorem}[Mean and Variance of TSS-Solve]
\label{thm:Var-CG}
Let $\mathbf{A}\in\mathbb{R}^{n\times n}$ be SPD with condition number $\kappa=\kappa(\mathbf{A})$, and consider $\mathbf{A}\mathbf{x}=\mathbf{y}$ with CG iterates $\{\mathbf{x}_i\}_{i=0}^{n}$ initialized at $\mathbf{x}_0=0$.
Let $\cdq$ be a random variable supported on $\{i_{\min},\ldots,i_{\max}\}$ with mass function $p_{\cdq}$, and define the TSS-Solve estimator
\begin{equation}
\label{eq:cdphissCG}
\widetilde{\mathbf{x}}_{\mathrm{tss}} \;=\; \frac{\mathbf{x}_{\cdq}-\mathbf{x}_{\cdq-1}}{\mathbb{P}(\cdq)} \;+\;  \mathbf{x}_{\imin-1}.
\end{equation}
Then
\begin{equation}
\label{eq:phiCGVar}
\mathbb{E}\big[\widetilde{\mathbf{x}}_{\mathrm{tss}}\big]=\mathbf{x}_{i_{\max}}
\qquad\text{and}\qquad
\mathrm{Var}\!\left(\widetilde{\mathbf{x}}_{\mathrm{tss}}\right) \;\le\; 16\,\cdf{\kappa}\,\|\mathbf{x}\|^{2}\;\Gamma_{\textrm{Solve}},
\end{equation}
where
\begin{equation}
\label{eq:Gamma1}    
\Gamma_{\textrm{Solve}} \;:=\; \sum_{j=i_{\min}}^{i_{\max}} \frac{\varrho^{2(j-1)}_{\textrm{Solve}}}{\mathbb{P}(\cdq=j)},
\qquad
\varrho_{\textrm{Solve}} \;:=\; \frac{\sqrt{\kappa}-1}{\sqrt{\kappa}+1}\in[0,1).
\end{equation}
\end{theorem}

\begin{proof}
Let $\Delta_j:=\mathbf{x}_j-\mathbf{x}_{j-1}$ and $\Delta_*:=\sum_{i=i_{\min}}^{i_{\max}}\Delta_i$.
From Proposition \ref{prop:TSSbasic}, we have
\[
\mathbb{E}\big[\widetilde{\mathbf{x}}_{\mathrm{tss}}\big]
=\sum_{i=1}^{i_{\max}}(\mathbf{x}_i-\mathbf{x}_{i-1})=\mathbf{x}_{i_{\max}}-\mathbf{x}_0=\mathbf{x}_{i_{\max}}.
\]
We next estimate the variance in \eqref{eq:phiCGVar}.  
Proposition \ref{prop:TSSbasic} shows that
\begin{equation}
\label{eq:vartssCG}
    \mathrm{Var}\!\left(\widetilde{\mathbf{x}}_{\mathrm{tss}}\right)\leq \sum_{j=\imin}^{\imax} \frac{||\Delta_j||^2}{\mathbb{P}(\cdq=j)}.
\end{equation}
% Then for $\cdq=j$, we have
% \begin{equation*}
%     \widetilde{\mathbf{x}}_{\mathrm{tss}}\big|_{\cdq=j}-\eebb(\widetilde{\mathbf{x}}_{\mathrm{tss}})= \frac{\Delta_j}{\mathbb{P}(\cdq=j)}-\sum_{i=\imin}^{\imax} \Delta_i = \frac{\Delta_j}{\mathbb{P}(\cdq=j)}-\Delta_*.
% \end{equation*}
% Thus, the variance 
% \begin{equation}
% \label{eq:varphi}
% \begin{aligned}
% \mathrm{Var}\!\left(\widetilde{\mathbf{x}}_{\mathrm{tss}}\right)
% &=\sum_{j=i_{\min}}^{i_{\max}}\mathbb{P}(\cdq=j)\left\|\frac{\Delta_j}{\mathbb{P}(\cdq=j)}-\Delta_*\right\|^2\\
%     &= \sum_{j=\imin}^{\imax} \frac{||\Delta_j||^2}{\mathbb{P}(\cdq=j)} -\sum_{j=\imin}^{\imax} 2\Delta_j\cdot\Delta_* + \sum_{j=\imin}^{\imax} \mathbb{P}(\cdq=j)||\Delta_*||^2 \\
%     &= \left( \sum_{j=\imin}^{\imax} \frac{||\Delta_j||^2}{\mathbb{P}(\cdq=j)}\right)  - 2||\Delta_*||^2+||\Delta_*||^2\\
%     &= \left(\sum_{j=\imin}^{\imax} \frac{||\Delta_j||^2}{\mathbb{P}(\cdq=j)}\right)  - ||\Delta_*||^2\leq \sum_{j=\imin}^{\imax} \frac{||\Delta_j||^2}{\mathbb{P}(\cdq=j)},
% \end{aligned}
% \end{equation}
% where we have used the fact that 
% $$\Delta_* = \sum_{j=\imin}^{\imax} \Delta_j \quad\text{and}\quad \sum_{j=\imin}^{\imax}\mathbb{P}(\cdq=j)=1.$$
We bound the term $||\Delta_j||^2=||\cdx_j-\cdx_{j-1}||^2$ using the condition number $\kappa$ as follows.
Recall the error estimate for CG in $l_2$ norm: 
\begin{equation*}
    ||\cdx_j-\cdx||\leq \cdf{\sqrt{\kappa}}\frac{2\varrho^j_{\textrm{Solve}}}{1+\varrho^{2j}_{\textrm{Solve}}} ||\cdx_0-\cdx||,\quad \text{with}\quad  
    \varrho_{\textrm{Solve}}=\frac{\sqrt{\kappa}-1}{\sqrt{\kappa}+1}<1.
\end{equation*}
We use this to deduce that 
\begin{equation*}
\begin{aligned}
    ||\cdx_j-\cdx_{j-1}||^2&\leq 2||\cdx_j-\cdx||^2+2||\cdx_{j-1}-\cdx||^2\\
    &\leq
    2\cdf{\kappa}\left(\frac{4\varrho^{2j}_{\textrm{Solve}}}{(1+\varrho^{2j}_{\textrm{Solve}})^2}+\frac{4\varrho^{2(j-1)}_{\textrm{Solve}}}{(1+\varrho^{2(j-1)}_{\textrm{Solve}})^2}\right) ||\cdx_0-\cdx||^2\\
    &\leq 2\cdf{\kappa}(4\varrho^{2(j-1)}_{\textrm{Solve}}+4\varrho^{2(j-1)}_{\textrm{Solve}})||\cdx_0-\cdx||^2 = 16\cdf{\kappa}\varrho^{2(j-1)}_{\textrm{Solve}}||\cdx_0-\cdx||^2.
\end{aligned}
\end{equation*}
Inserting this into \eqref{eq:vartssCG} yields 
\begin{equation*}
    \mathrm{Var}\!\left(\widetilde{\mathbf{x}}_{\mathrm{tss}}\right)\leq 16\cdf{\kappa}||\cdx_0-\cdx||^2\sum_{j=\imin}^{\imax} \frac{\varrho^{2(j-1)}_{\textrm{Solve}}}{\mathbb{P}(\cdq=j)},
\end{equation*}
which completes the proof under the assumption that $\cdx_0=0$.
\end{proof}

{We remark that the assumption $\mathbf{x}_0 = 0$ in Theorem~\ref{thm:Var-CG} is made without loss of generality; the result holds for arbitrary initial guesses. Indeed, if $\mathbf{x}_0 \neq 0$, one can define the TSS-Solve estimator as
\[
\widetilde{\mathbf{x}}_{\mathrm{tss}} = \frac{\Delta_Q}{\mathbb{P}(\cdq)} + \mathbf{x}_0 + \sum_{i=1}^{i_{\min}-1} \Delta_i,
\]
which is equivalent to the formulation in \eqref{eq:cdphissCG}, since both yield the same expression $\frac{\Delta_Q}{\mathbb{P}(\cdq)} + \mathbf{x}_{i_{\min}-1}$. The assumption $\mathbf{x}_0 = 0$ is adopted for notational convenience and simplifies the presentation of the estimator and its variance analysis.}

Theorem~\ref{thm:Var-CG} characterizes the mean and variance of the TSS--Solve estimator. First, the estimator is \emph{unbiased} relative to a fixed truncation level: its expectation matches the $i_{\max}$--step CG iterate, i.e., $\mathbb{E}[\widetilde{\mathbf{x}}_{\mathrm{tss}}] = \mathbf{x}_{i_{\max}}$. In the special case $i_{\max} = n$, the estimator becomes unbiased for the exact solution $\mathbf{x}$.
Second, the variance bound
\[
\mathrm{Var}(\widetilde{\mathbf{x}}_{\mathrm{tss}}) \;\le\; 16\,\cdf{\kappa}\,\|\mathbf{x}\|^2\,\Gamma_{\textrm{Solve}}
\]
demonstrates how both the conditioning of $\mathbf{A}$ and the sampling strategy influence the deviation of the estimator from the expectation. In particular, a smaller condition number $\kappa$ gives a smaller upper bound, suggesting that preconditioning is beneficial for variance reduction.

For fixed $\kappa$, the structure of the term
\[
\Gamma_{\textrm{Solve}} = \sum_{j=i_{\min}}^{i_{\max}} \frac{\varrho^{2(j-1)}_{\textrm{Solve}}}{\mathbb{P}(\cdq = j)}
\]
reveals the tradeoff between computational cost and statistical efficiency. Since \( \varrho^{2(j-1)}_{\textrm{Solve}} \) decays rapidly with \( j \), the later terms in the sum contribute less to the variance when they are weighted properly. A detailed discussion of the choice of the distribution $p_{\cdq}$ is presented in Section \ref{sub:discussion}.

The following corollary provides a concentration bound that quantifies the probability that the TSS-Solve estimator deviates from the $i_{\max}$-step CG iterate.

\begin{corollary}[Concentration of TSS-Solve]
\label{cor:phiCG}
Let $\widetilde{\mathbf{x}}_{\mathrm{tss}}$ denote the TSS-Solve estimator in \eqref{eq:cdphissCG}.
Under the assumptions of Theorem~\ref{thm:Var-CG}, for any $\delta>0$,
\[
\mathbb{P}\!\left(\,\big\|\widetilde{\mathbf{x}}_{\mathrm{tss}}-\mathbf{x}_{i_{\max}}\big\|\ge \delta\,\right)
\;\le\; 16\,\delta^{-2}\,\cdf{\kappa}\,\|\mathbf{x}\|^{2}\,\Gamma_{\textrm{Solve}},
\]
where $\Gamma_{\textrm{Solve}}$ is defined in \eqref{eq:Gamma1}.
\end{corollary}

\begin{proof}
 Since $\mathbb{E}[\widetilde{\mathbf{x}}_{\mathrm{tss}}] = \mathbf{x}_{i_{\max}}$ in \eqref{eq:phiCGVar}, it follows from Chebyshev's inequality that
    $$P\left(\Vert\widetilde{\mathbf{x}}_{\mathrm{tss}} - \mathbf{x}_{i_{\max}}\Vert\geq \delta\right)\leq \delta^{-2}\mathrm{Var}(\widetilde{\mathbf{x}}_{\mathrm{tss}}).$$
    The desired inequality follows by using the estimate for $\mathrm{Var}(\widetilde{\mathbf{x}}_{\mathrm{tss}})$ in \eqref{eq:phiCGVar}.
\end{proof}

\subsection{TSS-LogQF estimator}
\label{sub:TSS-Lanczos}
Having established the mean and variance properties of the TSS-Solve estimator for solving linear systems, we now consider the trace estimation problem for matrix logarithms, where a typical quantity of interest is the quadratic form \(\mathbf{z}^\top \log(\mathbf{A})\,\mathbf{z}\) with a SPD matrix $\mathbf{A}$ ($\zbf$ is a random vector). A standard approach is to approximate this quantity using Lanczos quadrature, which constructs a sequence of increasingly accurate estimates from Krylov subspace projections. In analogy with the TSS-Solve method, we define a truncated single-sample (TSS) variant of Lanczos quadrature that randomizes over a window of iterations to approximate the quadratic form for the matrix logarithm. We thus term the estimator TSS-LogQF. The following theorem establishes that the TSS-LogQF estimator is unbiased with respect to the deterministic Lanczos output truncated at a prescribed Krylov subspace dimension \(i_{\max}\), and provides an upper bound on its variance involving the condition number of \(\mathbf{A}\), the sampling range and the sampling distribution $p_{\cdq}$.
A more refined analysis of the variance bound and the appropriate choice of $p_{\cdq}$ will be presented in Section \ref{sub:discussion}.

\begin{theorem}[Mean and Variance of TSS-LogQF]
\label{thm:Var-logdet}
Let $\mathbf{A}\in\mathbb{R}^{n\times n}$ be SPD with condition number $\kappa$, and let $\mathbf{z}\in\mathbb{R}^n\!\setminus\!\{0\}$.
Consider the quadratic form
\[
\Phi \;=\; \frac{\mathbf{z}^\top \log(\mathbf{A})\,\mathbf{z}}{\|\mathbf{z}\|^2}.
\]
For a Lanczos run started at $\mathbf{q}_1=\mathbf{z}/\|\mathbf{z}\|$, let $\Tjz{j}$ be the $j\times j$ tridiagonal matrix after $j$ steps, and define
\[
s_j(\mathbf{z}) \;:=\; \mathbf{e}_1^\top \log\!\big(\Tjz{j}\big)\,\mathbf{e}_1,
\qquad
\Delta_j \;:=\; s_j(\mathbf{z})-s_{j-1}(\mathbf{z}),\qquad s_0(\mathbf{z})\equiv 0.
\]
With $\cdq$ supported on $\{i_{\min},\ldots,i_{\max}\}$, the TSS-LogQF estimator
\begin{equation}
\label{eq:philogdet}
\widetilde{\Phi}_{\mathrm{tss}} \;=\; s_{\imin-1}(\mathbf{z}) \;+\; \frac{\Delta_{\cdq}}{\mathbb{P}(\cdq)}
\end{equation}
satisfies
\begin{equation}
\label{eq:philogdetVar}
\mathbb{E}\big[\widetilde{\Phi}_{\mathrm{tss}}\big] \;=\; s_{i_{\max}}(\mathbf{z})
\qquad\text{and}\qquad
\mathrm{Var}\!\left(\widetilde{\Phi}_{\mathrm{tss}}\right)
\;\le\; 16\,\big(\sqrt{\kappa+1}+1\big)^2 \log^2(2\kappa)\;\Gamma_{\textrm{LogQF}},
\end{equation}
where
\begin{equation}
\label{eq:Gamma2}
\Gamma_{\textrm{LogQF}} \;:=\; \sum_{j=i_{\min}}^{i_{\max}} \frac{\varrho^{4(j-1)}_{\textrm{LogQF}}}{\mathbb{P}(\cdq=j)},
\qquad
\varrho_{\textrm{LogQF}} \;:=\; \frac{\sqrt{\kappa+1}-1}{\sqrt{\kappa+1}+1}\in[0,1).
\end{equation}
\end{theorem}

\begin{proof}
For $\mathbb{E}\big[\widetilde{\Phi}_{\mathrm{tss}}\big]$, it follows directly from \eqref{eq:TSS-expectation-again} and $s_0(\mathbf{z})=0$ that
$$\mathbb{E}\big[\widetilde{\Phi}_{\mathrm{tss}}\big]=\sum_{i=1}^{i_{\max}}\Delta_i=s_{i_{\max}}(\mathbf{z})-s_0(\zbf)=s_{i_{\max}}(\mathbf{z}).$$

    Next we estimate $\mathrm{Var}\!\left(\widetilde{\Phi}_{\mathrm{tss}}\right)
\;$. 
    Proposition \ref{prop:TSSbasic} shows that
    \begin{equation}
    \label{eq:varPhiDelta}
        \vvbb\!\left(\cdphiss\right)\leq \sum_{j=\imin}^{\imax} \frac{|\Delta_j|^2}{\mathbb{P}(\cdq=j)},
    \end{equation}
    where $\Delta_j$ is defined in \eqref{eq:philogdet} and is a scalar.
    To estimate $|\Delta_j|$, we first define $\Delta_{j,0} := s_j(\mathbf{z}) - \Phi$ as the (scaled) $j$-step Lanczos quadrature error.
    Since $\Delta_j=\Delta_{j,0}-\Delta_{j-1,0}$, we have
    \begin{equation}
    \label{eq:Deltaj0}
        |\Delta_j|^2\leq 2|\Delta_{j,0}|^2+2|\Delta_{j-1,0}|^2.
    \end{equation}
    According to \cite[Corollary 4]{kressner2022trace}, the Lanczos error $|\Delta_{j,0}|$ satisfies 
    $$|\Delta_{j,0}|\leq 2(\sqrt{\kappa+1}+1)\log(2\kappa)\varrho^{2j}_{\textrm{LogQF}}\quad\text{with}\quad\varrho_{\textrm{LogQF}}=\frac{\sqrt{\kappa+1}-1}{\sqrt{\kappa+1}+1}.$$
    Since $\varrho^j_{\textrm{LogQF}}\leq \varrho^{j-1}_{\textrm{LogQF}}$, we then deduce from \eqref{eq:Deltaj0} and the estimate above that 
    $$|\Delta_j|^2\leq 16(\sqrt{\kappa+1}+1)^2\log^2(2\kappa)\varrho^{4(j-1)}_{\textrm{LogQF}}.$$
Inserting this estimate into \eqref{eq:varPhiDelta} gives \eqref{eq:philogdetVar}.
\end{proof}

Theorem~\ref{thm:Var-logdet} describes the statistical properties of the TSS-LogQF estimator for log-determinant-type quantities. The estimator is unbiased with respect to the truncated quadrature: the expected value recovers the $i_{\max}$-step Lanczos approximation $s_{i_{\max}}(\mathbf{z})$ to the target quantity $\Phi$ in Theorem~\ref{thm:Var-logdet}. The variance bound reveals the influence of both the condition number $\kappa$ and the sampling distribution. The dependence on $\kappa$ enters through the logarithmic amplification factor $\log^2(2\kappa)$, the factor $(\sqrt{\kappa+1}+1)^2$ and the decay rate 
\[
\varrho_{\textrm{LogQF}} = \frac{\sqrt{\kappa+1}-1}{\sqrt{\kappa+1}+1},
\]
which governs the size of the increment terms $\Delta_j$. The quantity
\[
\Gamma_{\textrm{LogQF}} = \sum_{j=i_{\min}}^{i_{\max}} \frac{\varrho^{4(j-1)}_{\textrm{LogQF}}}{\mathbb{P}(\cdq = j)}
\]
quantifies how the sampling strategy interacts with this decay.

Theorem~\ref{thm:Var-logdet} also leads to a concentration bound for the TSS--LogQF estimator by applying Chebyshev’s inequality. 

\begin{corollary}[Concentration of TSS-LogQF]
Let $\widetilde{\Phi}_{\mathrm{tss}}$ denote the TSS Lanczos estimator in \eqref{eq:philogdet}.
Under the assumptions of Theorem \ref{thm:Var-logdet}, for any $\delta>0$,
\[
\mathbb{P}\!\left(\,\big|\widetilde{\Phi}_{\mathrm{tss}}-s_{i_{\max}}(\mathbf{z})\big|\ge \delta\,\right)
\;\le\; 16\,\delta^{-2}\,\big(\sqrt{\kappa+1}+1\big)^2 \log^2(2\kappa)\;\Gamma,
\]
where $\Gamma$ is defined in \eqref{eq:Gamma2}.
\end{corollary}
\begin{proof}
    Similar to the proof of Corollary \ref{cor:phiCG}, the inequality follows directly from Chebyshev's inequality and \eqref{eq:philogdetVar}.
\end{proof}

\subsection{Discussion on sampling distribution and condition number}
\label{sub:discussion}
In this section, we provide a detailed discussion of the key factors that govern the variance of the proposed TSS estimators. In particular, the variance bounds established in Theorem~\ref{thm:Var-CG} and Theorem~\ref{thm:Var-logdet} depend critically on the choice of the sampling distribution $p_{\cdq}$ for the random iteration count \( Q \) and the condition number \( \kappa \). 
% and the truncation interval \( [i_{\min}, i_{\max}] \).

Our analysis focuses on the structure and optimization of the \(\Gamma\) factor that appears in the variance upper bound. We reveal a delicate connection between the computational cost of the estimator and the variance through the lens of matrix conditioning. Since the TSS-Solve and TSS-LogQF estimators possess $\Gamma$ factors with distinct progression patterns, we denote their associated \(\Gamma\) factors as follows:
\begin{align}
    \label{eq:GammaSolve}
    \Gamma_{\textrm{Solve}} &= \sum_{j=i_{\min}}^{i_{\max}} \frac{\varrho^{2(j-1)}_{\textrm{Solve}}}{\mathbb{P}(\cdq = j)},
    \quad \text{with} \quad
    \varrho_{\textrm{Solve}} = \frac{\sqrt{\kappa} - 1}{\sqrt{\kappa} + 1}, \\
    \label{eq:GammaLogQF}
    \Gamma_{\textrm{LogQF}} &= \sum_{j=i_{\min}}^{i_{\max}} \frac{\varrho_{\textrm{LogQF}}^{4(j-1)}}{\mathbb{P}(\cdq = j)},
    \quad \text{with} \quad
    \varrho_{\textrm{LogQF}} = \frac{\sqrt{\kappa + 1} - 1}{\sqrt{\kappa + 1} + 1}.
\end{align}
Both expressions are weighted geometric sums with weights given by inverse sampling probabilities \( \mathbb{P}(\cdq = j) \), underlining the importance of the sampling distribution $p_{\cdq}$ in controlling variance and computational cost.

Since the sampling distribution $p_{\cdq}$ is user-specified, it is natural to ask how to choose it to reduce variance of the estimator. In particular, we seek a distribution that minimizes the corresponding \(\Gamma\) factor in the variance upper bound.

Theorem~\ref{thm:bestP} provides an explicit solution of the sampling distribution $p_{\cdq}$ that minimizes \(\Gamma\). The following lemma is essential to solving the minimization problem.

\begin{lemma}
\label{lm:minimization}
Let \( m_2 > m_1 \) be positive integers and \( t > 0 \). Consider the constrained minimization
\begin{equation}
\label{eq:minpi}
\begin{aligned}
    \min_{\{p_i\}_{i=m_1}^{m_2}} \quad &\sum_{i=m_1}^{m_2} \frac{t^i}{p_i} \qquad
    \textrm{s.t.} \quad \sum\limits_{i=m_1}^{m_2} p_i = 1, \quad p_i > 0.
\end{aligned}
\end{equation}
The minimum is achieved at
\begin{equation}
\label{eq:bestpi}
    p_i = \frac{t^{i/2}}{\sum\limits_{l=m_1}^{m_2} t^{l/2}}, \qquad i=m_1,\dots,m_2,
\end{equation}
with minimal value
\begin{equation}
\label{eq:minvalue}
   \frac{t^{m_1} \left(t^{(m_2 - m_1 + 1)/2} - 1\right)^2}{(t^{1/2} - 1)^2}.
\end{equation}
\end{lemma}
\begin{proof}
    Define $a_i=\sqrt{\frac{t^i}{p_i}}$ and $b_i=\sqrt{p_i}$.
    The sum in \eqref{eq:minpi} can be written as 
    $$ \sum_{i=m_1}^{m_2} \frac{t^i}{p_i} = \left( \sum_{i=m_1}^{m_2} a_i^2 \right)\cdot 1 = \left( \sum_{i=m_1}^{m_2} a_i^2 \right)\left( \sum_{i=m_1}^{m_2} p_i\right)=\left( \sum_{i=m_1}^{m_2} a_i^2 \right)\left( \sum_{i=m_1}^{m_2} b_i^2\right).$$
    It follows from Cauchy--Schwarz inequality 
    $\left(\sum\limits_i a_i^2\right)\left(\sum\limits_i b_i^2\right)\geq \left(\sum\limits_i a_i b_i\right)^2$
    that 
    \begin{equation*}
        \sum_{i=m_1}^{m_2} \frac{t^i}{p_i}\geq \left( \sum_{i=m_1}^{m_2} a_i b_i\right)^2 = \left( \sum_{i=m_1}^{m_2} t^{i/2}\right)^2=\frac{t^{m_1}(t^{(m_2-m_1+1)/2}-1)^2}{(t^{1/2}-1)^2},
    \end{equation*}
    which is independent of the choice of $p_i$'s.
    The minimum is achieved when there is a constant $c$ such that $a_i=cb_i$ for all $i$, i.e. $\sqrt{\frac{t^i}{p_i}}=c\sqrt{p_i}$,
    or equivalently, $p_i=c^{-1}t^{i/2}.$
    Using the constraint $\sum\limits_{i=m_1}^{m_2}p_i=1$, we compute that $c$ is the normalization constant:
    $$c = \sum_{i=m_1}^{m_2}t^{i/2}.$$
    Therefore, we conclude that the value in \eqref{eq:minvalue} is the minimum and can be obtained by choosing 
    $p_i=c^{-1}t^{i/2}$ with $c$ defined above. 
    This proves \eqref{eq:minvalue} and \eqref{eq:bestpi}.
\end{proof}

\begin{theorem}
\label{thm:bestP}
For $\Gamma_{\emph{Solve}}$ in \eqref{eq:GammaSolve} and $\Gamma_{\emph{LogQF}}$ in \eqref{eq:GammaLogQF}, the minimum over all possible sampling distribution $p_{\cdq}$ is as follows.
\begin{equation}
\label{eq:minGamma}
\begin{aligned}
    \min_{p_{\cdq}} \Gamma_{\emph{Solve}} &= \frac{\varrho_{\emph{Solve}}^{2(\imin-1)} \left(\varrho_{\emph{Solve}}^{\imax-\imin+1} - 1\right)^2}{(\varrho_{\emph{Solve}} - 1)^2};
    \\
    \min_{p_Q} \Gamma_{\emph{LogQF}} &= 
    \frac{\varrho_{\emph{LogQF}}^{4(\imin-1)} \left(\varrho_{\emph{LogQF}}^{2(\imax-\imin+1)} - 1\right)^2}{(\varrho_{\emph{LogQF}}^2 - 1)^2}.
\end{aligned}
\end{equation}
$\min_{p_Q}\Gamma_{\emph{Solve}}$ is achieved at the distribution
\begin{equation}
\label{eq:pSolve-opt}
    \mathbb{P}(Q=i) = \frac{\varrho_{\emph{Solve}}^i}{\sum\limits_{l=\imin}^{\imax} \varrho_{\emph{Solve}}^l}=\frac{\left(\tfrac{\sqrt{\kappa} - 1}{\sqrt{\kappa} + 1}\right)^i}{\sum\limits_{l=\imin}^{\imax} \left(\tfrac{\sqrt{\kappa} - 1}{\sqrt{\kappa} + 1}\right)^l}\qquad (\imin\leq i \leq \imax);
\end{equation}
$\min_{p_Q}\Gamma_{\emph{LogQF}}$ is achieved at the distribution
\begin{equation}
\label{eq:pLogQF-opt}
    \mathbb{P}(Q=i) = \frac{\varrho_{\emph{LogQF}}^{2i}}{\sum\limits_{l=\imin}^{\imax} \varrho_{\emph{LogQF}}^{2l}}=
\frac{\left(\tfrac{\sqrt{\kappa + 1} - 1}{\sqrt{\kappa + 1} + 1}\right)^{2i}}{\sum\limits_{l=\imin}^{\imax} \left(\tfrac{\sqrt{\kappa + 1} - 1}{\sqrt{\kappa + 1} + 1}\right)^{2l}}\qquad (\imin\leq i\leq \imax).
\end{equation}
\end{theorem}
\begin{proof}
    The result follows directly by applying Lemma \ref{lm:minimization} to \eqref{eq:GammaSolve} with $t=\varrho_{\emph{Solve}}^2$ in Lemma \ref{lm:minimization} and to \eqref{eq:GammaLogQF} with
    $t=\varrho_{\emph{LogQF}}^4$ in Lemma \ref{lm:minimization}.
\end{proof}

We call the distributions in \eqref{eq:pSolve-opt} and \eqref{eq:pLogQF-opt} \emph{$\Gamma$-optimal} (for the variance upper bounds), since they minimize the corresponding $\Gamma$ factors for fixed $(i_{\min}, i_{\max})$ and a given conditioning $\kappa$.
As illustrated in Figure~\ref{fig:GammaCompare}, the $\Gamma$-optimal choice can reduce $\Gamma$ by orders of magnitude compared with commonly used geometric distributions that do not account for conditioning~\cite{pmlr-v139-potapczynski21a}.

{The $\Gamma$-optimal distributions \eqref{eq:pSolve-opt}--\eqref{eq:pLogQF-opt} depend on $\kappa$ and minimize the \emph{variance upper bound} through the corresponding $\Gamma$ factor.
We therefore view them primarily as (i) a theoretical benchmark and (ii) guidance on how the sampling tail should scale with conditioning.
In practice, the effective condition number of the (preconditioned) operator may be unknown a priori and may vary across hyperparameters, so plugging in an inexact $\widehat{\kappa}$ yields a distribution that is not exactly $\Gamma$-optimal.
The consequence of such mismatch is typically a tail that is slightly too light or too heavy; empirically, once preconditioning is reasonably effective, the estimator variance becomes much less sensitive to the precise choice of $p_Q$, so simple fixed geometric distributions are often sufficient. When the preconditioner is poor (or absent) and one wishes to tune $p_Q$, a rough condition surrogate can be obtained with essentially no additional MVP cost from the same Krylov/Lanczos run: the extremal Ritz values of the small tridiagonal $\mathbf T_j$ provide estimates $\widehat{\lambda}_{\min}$ and $\widehat{\lambda}_{\max}$, hence $\widehat{\kappa}\approx \widehat{\lambda}_{\max}/\widehat{\lambda}_{\min}$.
One may then either plug $\widehat{\kappa}$ into \eqref{eq:pSolve-opt}--\eqref{eq:pLogQF-opt} or simply tune the decay rate of a geometric distribution.
}

With the optimal solution for the $\Gamma$ factor, as shown below, we can optimize the variance estimates for the TSS-Solve and TSS-LogQF estimators and derive upper bounds that only involve the condition number $\kappa$ and the truncation parameters $\imin,\imax$.

\begin{figure}
    \centering
    \includegraphics[width=0.5\linewidth]{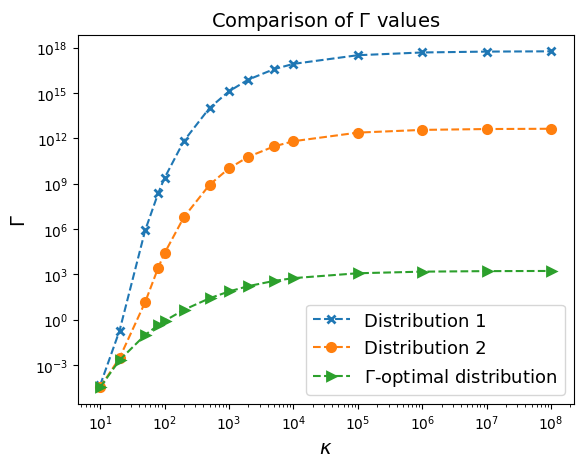}
    \caption{Comparison of $\Gamma_{\textrm{Solve}}$ values in \eqref{eq:GammaSolve} for different distributions $p_Q$: two commonly used distributions $\mathbb{P}_1(Q=j)\propto e^{-0.5j}$ and $\mathbb{P}_2(Q=j)\propto 2^{-j}$ versus the $\Gamma$-optimal distribution in \eqref{eq:pSolve-opt}.}
    \label{fig:GammaCompare}
\end{figure}

\begin{theorem}
\label{thm:VarOptimal}
    Denote $i_1=\imin$ and $i_2=\imax$.
    With the $\Gamma$-optimal choice of $p_Q$ in Theorem \ref{thm:bestP} for the TSS-Solve and TSS-LogQF estimators, the respective variance satisfies 
    \begin{equation*}
    \begin{aligned}
        \vvbb\big(\cdphiss^{\emph{(Solve)}}\big) \;& \leq\; 4\,\cdf{\kappa}\,\|\mathbf{x}\|^{2}\; {\left(\tfrac{\sqrt{\kappa} - 1}{\sqrt{\kappa} + 1}\right)^{2(i_1-1)} \left(\left(\tfrac{\sqrt{\kappa} - 1}{\sqrt{\kappa} + 1}\right)^{i_2-i_1+1} - 1\right)^2}(\sqrt{\kappa}+1)^2 \\
        \vvbb\big(\cdphiss^{\emph{(LogQF)}}\big) \;& \leq \; \tfrac{\big(\sqrt{\kappa+1}+1\big)^6\log^2(2\kappa)}{(\kappa+1)} \;{\left(\tfrac{\sqrt{\kappa + 1} - 1}{\sqrt{\kappa + 1} + 1}\right)^{4(i_1-1)} \left(\left(\tfrac{\sqrt{\kappa + 1} - 1}{\sqrt{\kappa + 1} + 1}\right)^{2(i_2-i_1+1)} - 1\right)^2}
    \end{aligned}
    \end{equation*}
\end{theorem}
\begin{proof}
This can be proved by inserting the optimal $\Gamma$-values in Theorem \ref{thm:bestP} into the variance estimates in \eqref{eq:phiCGVar} and \eqref{eq:philogdetVar}, respectively, where
the denominators in \eqref{eq:minGamma}, i.e.
$$\left(\frac{\sqrt{\kappa} - 1}{\sqrt{\kappa} + 1}- 1\right)^2
\quad\text{and}\quad 
\left(\left(\frac{\sqrt{\kappa + 1} - 1}{\sqrt{\kappa + 1} + 1}\right)^2 - 1\right)^2,$$ 
can be simplified as
$$
\frac{4}{(\sqrt{\kappa}+1)^2}
\quad\text{and}\quad
\frac{16(\kappa+1)}{\left(\sqrt{\kappa+1}+1\right)^4}.
$$
The proof is complete.
\end{proof}

To obtain a more straightforward dependence on $\kappa$ for the bounds in Theorem \ref{thm:VarOptimal}, we consider the asymptotic behavior of the bounds as $\kappa\to\infty$, indicating the worst-case scenario for ill-conditioned problems. \cdf{The result is stated below.}
%The analysis shows that $\Gamma=\mathcal{O}(1)$ as $\kappa\to\infty$ and the variance of TSS estimators grows at most \cdf{linearly} or nearly linearly with $\kappa$.

\begin{theorem}[$\kappa$-asymptotic estimates]
    Let $\Gamma_{\emph{Solve}}$ and $\Gamma_{\emph{LogQF}}$ denote the optimal $\Gamma$ values in \eqref{eq:minGamma}.
    Then as $\kappa\to\infty$,
    $$\Gamma_{\emph{Solve}} = \mathcal{O}(1)\quad\text{and}\quad \Gamma_{\emph{LogQF}} = \mathcal{O}(1).$$
    Furthermore, 
    % for the $\Gamma$-optimal variance bounds in Theorem \ref{thm:VarOptimal}, 
    as $\kappa\to\infty$,
    we have 
    \begin{equation}
    \label{eq:VarOK}
    \vvbb\big(\cdphiss^{\emph{(Solve)}}\big) = \mathcal{O}(\cdf{\kappa})\quad\text{and}\quad \vvbb\big(\cdphiss^{\emph{(LogQF)}}\big) = \mathcal{O}(\kappa\log^2\kappa).
    \end{equation}
\end{theorem}
\begin{proof}
    % For $\Gamma_{\textrm{Solve}}$, 
    As $\kappa\to\infty$, we have
$$\varrho_{\textrm{Solve}}^{m} = \left(\frac{\sqrt{\kappa} - 1}{\sqrt{\kappa} + 1}\right)^{m}
=\left(1-\frac{2}{\sqrt{\kappa}+1}\right)^{m} 
\approx 1-\frac{2m}{\sqrt{\kappa}+1}.$$
Therefore,
\begin{equation*}
\begin{aligned}
    \left(\frac{\sqrt{\kappa} - 1}{\sqrt{\kappa} + 1}\right)^{2(\imin-1)} 
    &\approx 1-\frac{4(\imin-1)}{\sqrt{\kappa}+1},\\
\left(\frac{\sqrt{\kappa} - 1}{\sqrt{\kappa} + 1}\right)^{\imax-\imin+1} 
&\approx 1-\frac{2(\imax-\imin+1)}{\sqrt{\kappa}+1}.
\end{aligned}
\end{equation*}
Then it can be computed that, as $\kappa\to\infty$, 
$$\Gamma_{\textrm{Solve}} = \mathcal{O}\left( \frac{1\cdot \frac{1}{(\sqrt{\kappa}+1)^2}}{\frac{1}{(\sqrt{\kappa}+1)^2}}\right)=\mathcal{O}(1).$$
Similarly, as $\kappa\to\infty$,
$$\varrho_{\textrm{LogQF}}^m=\left(\frac{\sqrt{\kappa + 1} - 1}{\sqrt{\kappa + 1} + 1}\right)^{m} \approx 1 - \frac{2m}{\sqrt{\kappa + 1} + 1}.$$
The same analysis yields that 
$$\Gamma_{\textrm{LogQF}} = \mathcal{O}(1).$$
Since the $\Gamma$ factors are both $\mathcal{O}(1)$,
the asymptotic estimates for the variance bounds of the two estimators follow immediately from \eqref{eq:phiCGVar} and \eqref{eq:philogdetVar} where $\Gamma$ is a prefactor to the $\kappa$ terms. 
The proof is complete.
\end{proof}

\begin{figure}
    \centering
    \includegraphics[width=0.46\linewidth]{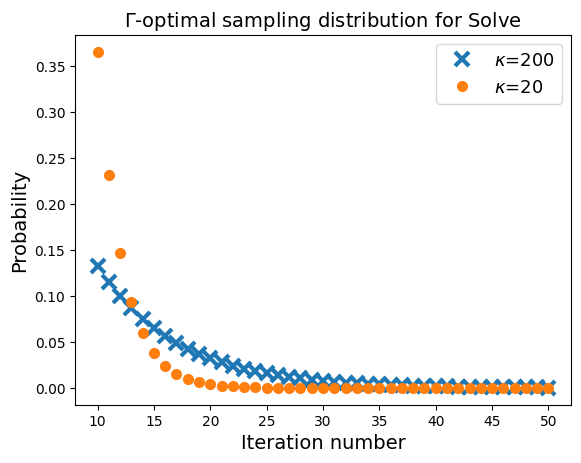}
    \includegraphics[width=0.46\linewidth]{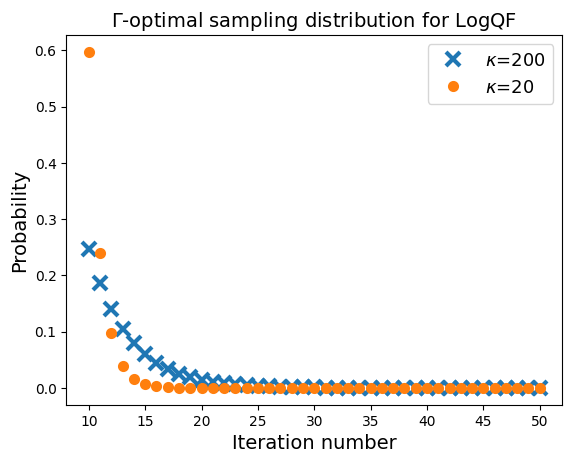}
    \caption{$\Gamma$-optimal sampling distribution has lighter tail for smaller condition number $\kappa$. Left: TSS-Solve; Right: TSS-LogQF.}
    \label{fig:p-curve-kappa}
\end{figure}

\paragraph{Impact of conditioning}
From the theoretical results in Section \ref{sub:TSS-CG} and Section \ref{sub:TSS-Lanczos}, we see that a smaller condition number $\kappa$ helps to reduce the variance of the underlying stochastic estimator.
On the other hand, Theorem \ref{thm:bestP} implies that conditioning is important for reducing the number of iterations.
To see this, consider the $\Gamma$-optimal choice of $p_Q$. 
It can be seen from Theorem \ref{thm:bestP} that smaller $\kappa$ yields smaller $\varrho$, which will then produce smaller tails in $p_Q$, namely, smaller $\mathbb{P}(Q=j)$ for large $j$'s. 
This means that it is more likely to sample a smaller iteration number when the condition number $\kappa$ is smaller.
An example is shown in Figure \ref{fig:p-curve-kappa} where we compare the $\Gamma$-optimal distribution under two different condition numbers $\kappa=20$ (orange dots) and $\kappa=200$ (blue `x'). It is easy to see that for both TSS-Solve and TSS-LogQF, $\kappa=20$ produces a much higher chance to sample small iteration numbers compared to $\kappa=200$. 
Therefore, we see that conditioning is important for not only a variance-reduced estimator, but also lower computational cost.

\section{Numerically stable implementation of preconditioned TSS estimators}
\label{subsec:stable-precond-lanczos}

The TSS estimators introduced in Section~\ref{sec:tss} are effective only when the underlying Krylov recurrences remain numerically stable. In finite precision arithmetic,  CG may suffer from loss of $\mathbf{A}$-orthogonality and inaccurate recurrence coefficients. These instabilities degrade the quality of the incremental updates $(\mathbf{x}_i - \mathbf{x}_{i-1})$ used by TSS-Solve and the tridiagonal matrices $\mathbf{T}_j$ used in Lanczos quadrature. Consequently, the effectiveness of TSS estimators can be severely compromised in ill-conditioned problems. {Although CG and Lanczos are mathematically equivalent in exact arithmetic, their finite-precision behavior differs in practice. The tridiagonal matrix $\mathbf T_m$ can be reconstructed from CG recurrence coefficients, as is commonly done in GP software; however, loss of orthogonality in the  Krylov basis may lead to inaccurate spectral approximations. For this reason, our PTSS implementation employs an explicit Lanczos process with full or windowed reorthogonalization to robustly form the tridiagonal matrices used in quadrature. Importantly, preconditioning reduces the effective condition number and keeps iteration counts modest, making reorthogonalization affordable in the intended regime.
}

We illustrate this issue using a simple 1D kernel regression example. Consider $100$ data points randomly sampled from $[0, 100]$ and a Gaussian kernel with length-scale $l=1.0$, noise variance $\mu = 0.001$, and scaling factor $f=1.0$. Let $\widehat{\mathbf{K}}$ denote the resulting kernel matrix. To evaluate solver behavior, we solve the system $\widehat{\mathbf{K}}\mathbf{x} = \mathbf{y}$ with $\mathbf{y}$ drawn uniformly randomly from $[-0.5, 0.5]^n$ using both {CG and Lanczos with full reorthogonalization.}

\begin{figure}[htbp]
    \centering
    \includegraphics[width=0.95\linewidth]{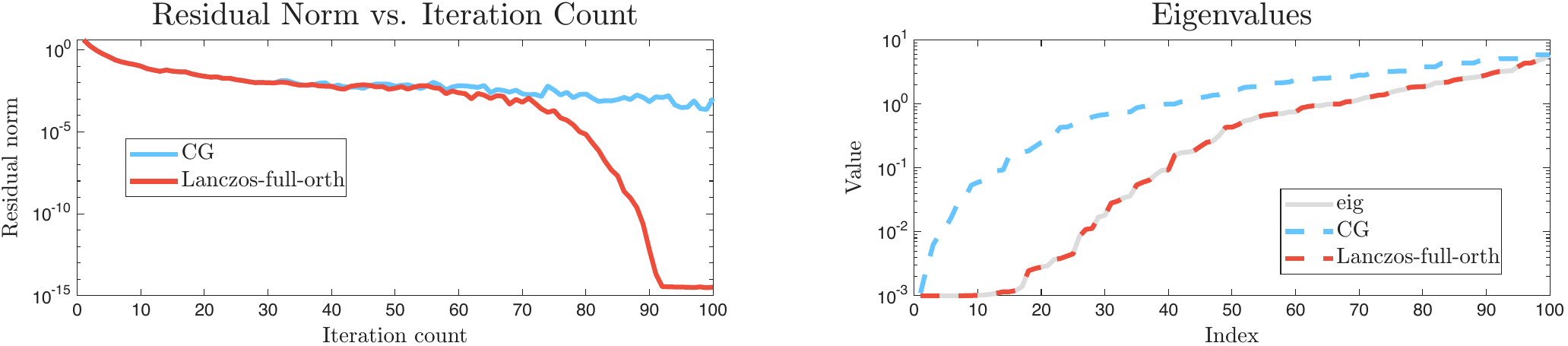}
    \caption{CG vs Lanczos with full reorthogonalization for solving linear systems with $\widehat{\mathbf{K}}$ and for solving eigenvalue problems. The tests use a Gaussian kernel with $\mu=0.001$, $l=1.0$ and $f=1.0$ on a simple 1D dataset with $100$ points randomly distributed in $[0,100]$. The eigenvalue approximations are obtained from the Ritz values of $\mathbf T_{100}$ generated by the CG and full-reorthogonalization Lanczos algorithms.
    }
    \label{fig:cg gmres lanc}
\end{figure}

Figure~\ref{fig:cg gmres lanc} (left) shows that, without preconditioning, CG fails to reach an acceptable residual norm even after 100 iterations. In contrast, the Lanczos method with full orthogonalization achieves substantially better performance. The tridiagonal matrix $\mathbf{T}_{100}$ extracted from CG also fails to accurately represent the spectrum of $\widehat{\mathbf{K}}$, as illustrated in Figure~\ref{fig:cg gmres lanc} (right). These discrepancies persist despite the small problem size, highlighting that although CG-based estimators are theoretically unbiased, they can become unreliable in finite precision due to orthogonality loss. 

This example demonstrates that, in addition to its theoretical role in variance reduction, preconditioning significantly accelerates the convergence of the Lanczos iterations by improving the conditioning of the matrix. This enables the estimator to achieve high accuracy within a small number of steps, effectively minimizing the computational overhead. Consequently, we utilize full reorthogonalization in most experiments as the iteration counts remain low, while also evaluating the windowed variant (which reorthogonalizes against a small window of previous vectors) in specific comparisons.

{In the implementation, we store the Lanczos recurrence coefficients so that $\mathbf T_j$ can be formed for any $j$ at negligible additional cost. For the PTSS estimators used in our main experiments, only a \emph{constant} number of evaluations per probe are required: $s_Q$ and $s_{Q-1}$ (to form $\Delta_Q$), together with the deterministic offset $s_{i_{\min}-1}$. We therefore do \emph{not} compute $s_j = \mathbf e_1^\top f(\mathbf T_j)\mathbf e_1$ for all $j \in [i_{\min}, i_{\max}]$. Each evaluation involves a small $j \times j$ symmetric tridiagonal matrix and costs $\mathcal O(j^3)$ via a standard tridiagonal eigensolver. Since $j$ is modest (typically in the tens), this cost is negligible compared to the dominant matrix–vector products. For variants requiring multiple $j$ values, the same small-tridiagonal cost argument applies; more elaborate incremental updates are unnecessary in our regime.}

\section{Numerical experiments}\label{sec: numerical experiments}

This section evaluates the proposed TSS estimator. We evaluate the Truncated Single-Sample (TSS) estimator in three roles: (a) as a standalone estimator for quantities; (b) as an estimator for Gaussian process (GP) negative log marginal likelihood and its gradients; and (c) embedded in GP hyperparameter optimization.
To isolate where the gains come from, we consider three controlled factors: (a) preconditioning: none (NP) versus AFN-preconditioned (AFN) \cite{Zhao_Xu_Huang_Chow_Xi_2024}; (b) truncation strategy: naive truncated (T) versus TSS; and (c) reorthogonalization window size $i_{\text{orth}}$. At iteration $k$, we reorthogonalize the newly generated Krylov vector against the last $\min\{i_{\text{orth}}-1,\,k-1\}$ previously stored Krylov basis vectors. This yields multiple configurations, denoted as [precond]-[truncation]-[$i_{\text{orth}}$]. For example, NP-T-$n$ denotes the unpreconditioned truncated Lanczos estimator with full reorthogonalization window.
We implemented our methods in \texttt{MATLAB}, and all experiments were performed using \texttt{MATLAB} R2025a.

\subsection{Preconditioned TSS estimator}

We begin our experiments with a simple inverse quadratic form task. Specifically, we estimate the term $\mathbf{y}^\top\widehat{\mathbf{K}}^{-1}\mathbf{y}$ and $\log|\widehat{\mathbf{K}}|$ where $\widehat{\mathbf{K}}$ is a regularized kernel matrix $f^2(\mathbf{K}+\mu \mathbf{I})$ generated with 4,096 points sampled uniformly from the cube $[0,16]^3$, and $\mathbf{y}$ is a vector with i.i.d.\ entries in $[-0.5,0.5]$.
We fix $f=1.0$, $\mu=0.01$, and test with the RBF kernel $\kappa(\mathbf{x},\mathbf{y})=\exp(-||\mathbf{x}-\mathbf{y}||_2^2/2l^2)$ with multiple $l$.
We sweep $l$ over a ``middle-rank'' regime $[1,10]$ in which the eigenvalues of $\widehat{\mathbf{K}}$ decay toward $f^2\mu$ while leaving many intermediate-magnitude eigenvalues, making the resulting linear systems challenging for unpreconditioned CG.

{In all experiments, the truncation window $[i_{\min}, i_{\max}]$ is chosen to reflect a practical Krylov budget rather than the full dimension $n$.
The upper bound $i_{\max}$ represents the largest iteration count that would be considered affordable for a deterministic truncated method to achieve an acceptable accuracy. %under the same matrix–vector product (MVP) budget.
In the preconditioned regime (e.g., with AFN), the effective condition number is significantly reduced, so a modest depth (typically 10–20 iterations) already yields a high-quality approximation.
The lower bound $i_{\min}$ is selected to avoid the high-variance “burn-in” phase of the Krylov process (empirically around 5–10 iterations).
Thus, when $i_{\max}\ll n$, the estimator remains biased relative to the exact solution, but TSS matches the mean of the $i_{\max}$-step deterministic approximation while reducing the expected computational cost.} 

{To make comparisons transparent, we report computational cost primarily in MVPs with the dense kernel matrix.
A deterministic truncated method with $m$ iterations uses exactly $m$ MVPs per probe,
whereas PTSS uses a random number of iterations $Q$ and thus costs $\mathbb E[Q]$ MVPs in expectation.
Because deterministic baselines require an integer iteration budget, we compare PTSS against $m=\lceil \mathbb E[Q]\rceil$ as the closest fixed-cost baseline when appropriate.
}

\subsubsection{TSS vs naive truncation}

In this first experiment, we compare the performance of the deterministic AFN-preconditioned (AFN-T-$n$) estimator with that of the stochastic AFN-preconditioned TSS (AFN-TSS-$n$) to evaluate the effectiveness of TSS in inverse quadratic form $\mathcal{Q}$ estimation. 
We fix both the rank and Schur complement fill level for AFN to 32.
For TSS, we choose $\mathbb{P}({Q}=j)\propto e^{-0.5j}$ in this experiment and set $i_{\min}=5$ and $i_{\max}=10$. For each length-scale we estimate 10,000 times and report the mean.
We also report the error obtained via AFN-T-$n$ with different truncations, including $i_{\min}$, $\lceil\mathbb{E}({Q})\rceil$ (rounded up), and $i_{\max}$.
Note that the expected cost of AFN-TSS-$n$ is close to, but lower than AFN-T-$n$ with $\mathbb{E}({Q})$ iterations.

\begin{figure}[htbp]
    \centering
    \includegraphics[width=0.95\linewidth]{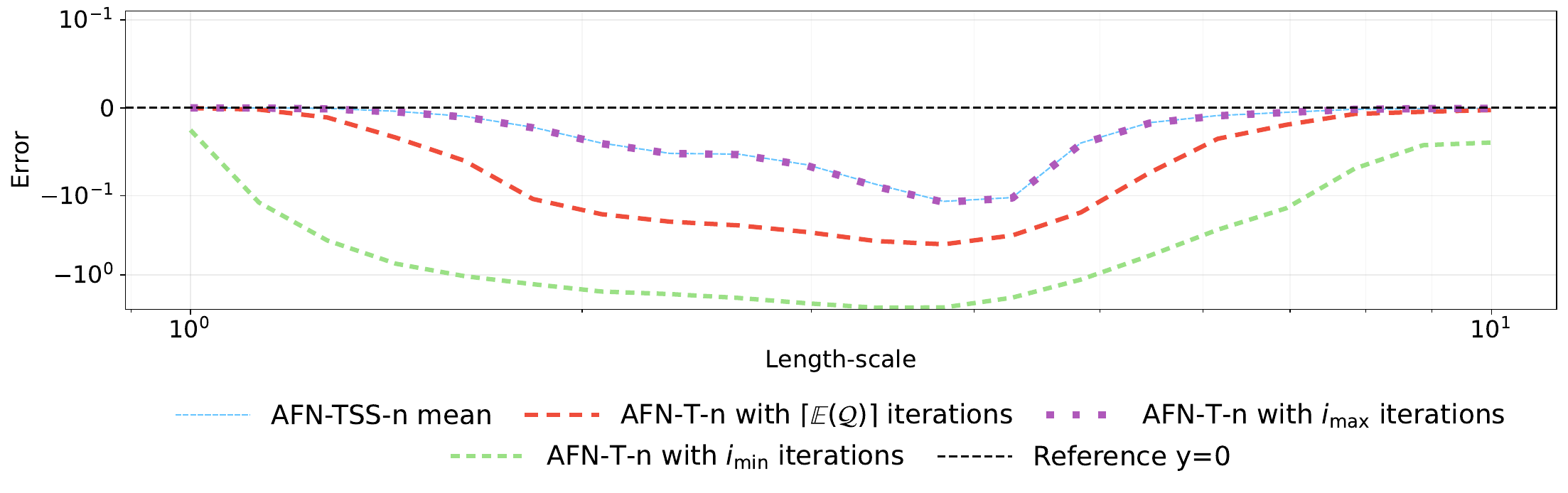}
    \caption{Signed mean error of the quadratic form estimation $\mathbf y^\top \widehat{\mathbf K}^{-1}\mathbf y$ versus length-scale $l$ (RBF kernel with $f=1.0$, $\mu=10^{-2}$). We use $\mathbb{P}({Q}=j)\propto e^{-0.5j}$ with $i_{\min}=5$ and $i_{\max}=10$ for TSS estimator, fix both the rank and Schur complement fill level for AFN to 32, and sample 10,000 times for each $l$ to report the mean. We also include the results for AFN-T-$n$ with $i_{\min}$, $i_{\max}$, and $\lceil\mathbb{E}({Q})\rceil=8$ iterations. Note that the y-axis uses a symmetric log scale to display both positive and negative errors, which is linear for values near zero.}
    \label{fig:quad comphensive}
\end{figure}

Figure~\ref{fig:quad comphensive} plots the signed error versus the length-scale $l$ defined as the difference between the estimation and the true value $\mathbf y^\top \widehat{\mathbf K}^{-1}\mathbf y$. The same probe vector y is used across all length scales. The results clearly demonstrate that all the deterministic truncated estimators are systematically biased and consistently underestimate the value. As expected, the error (bias) is most significant for the cheapest estimator with $i_{\min}$ iterations and remains large for the estimator with $\lceil\mathbb{E}({Q})\rceil$ iterations.
The AFN-TSS-$n$ estimator, whose computational cost is comparable to that of the truncated Lanczos method with $\lceil\mathbb{E}({Q})\rceil$ iterations, achieves mean errors close to those of the more expensive truncated Lanczos with $i_{\max}$ iterations across all $l$. This verifies the computational benefits of the TSS estimator.

\subsubsection{Preconditioned vs unpreconditioned}

In the next experiments, we examine how preconditioning changes the bias-variance trade-off, and how the sampling distribution $p_{\cdq}$ influences the performance of the estimator. 
We solve the same inverse quadratic form task with the RBF kernel as done in Section 5.1.1 and compare NP-TSS-$n$ with the AFN-preconditioned AFN-TSS-$n$ using the same setting.
In this test, we set the rank and Schur complement fill level to 64, and fix $i_{\min}=5$, $i_{\max}=15$.
We compare three different distributions: $\mathbb{P}({Q}=j)\propto e^{-0.5j}$, $\mathbb{P}({Q}=j)\propto 2^{-j}$, and the $\Gamma$-optimal distribution as described in \eqref{eq:pSolve-opt}.

\begin{figure}[htbp]
    \centering
    \includegraphics[width=0.95\linewidth]{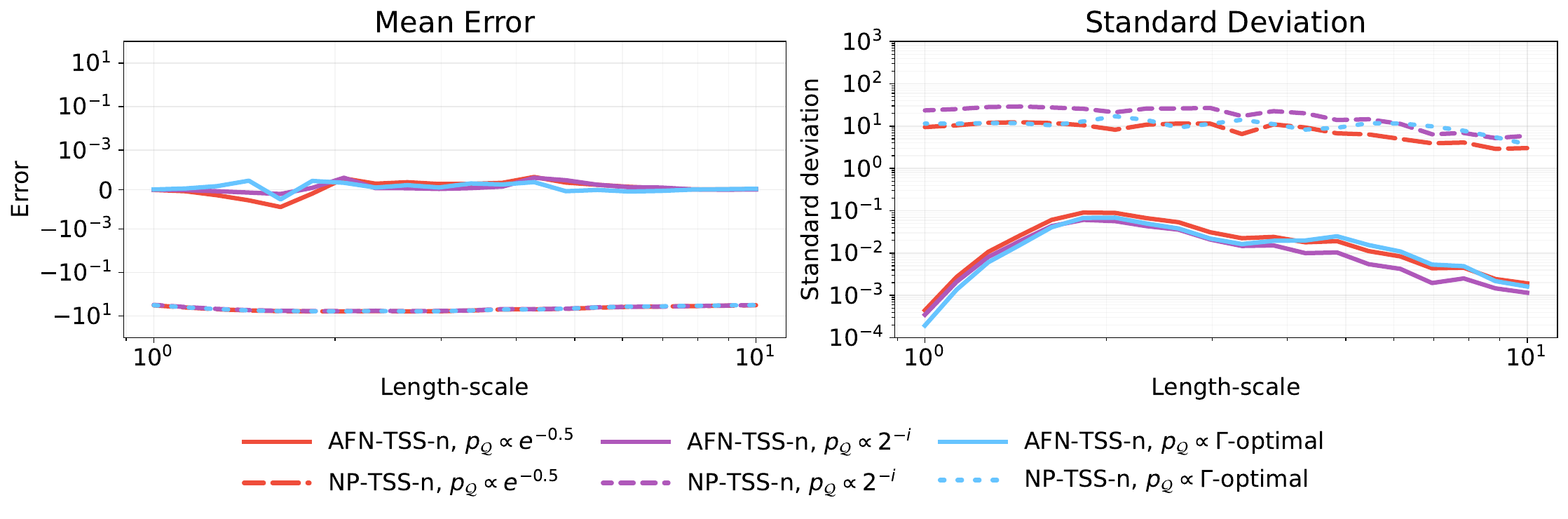}
    \caption{Signed mean error of the inverse quadratic form estimation $\mathbf y^\top \widehat{\mathbf K}^{-1}\mathbf y$ versus length-scale $l$ (RBF kernel with $f=1.0$, $\mu=10^{-2}$). The AFN-TSS-$n$ is compared with NP-TSS-$n$ with $\mathbb{P}({Q}=j)\propto e^{-0.5j}$, $\mathbb{P}({Q}=j)\propto 2^{-j}$, and the condition number-based distribution as in \eqref{eq:pSolve-opt} using $i_{\min}=5$ and $i_{\max}=15$. We fix both the rank and Schur complement fill level for AFN to 64, and sample 10,000 times for each $l$ to report the mean (left) and the standard deviation (right).}
    \label{fig: different distribution}
\end{figure}

Figure~\ref{fig: different distribution} plots the results of these experiments.
First, comparing AFN-TSS-$n$ with NP-TSS-$n$ shows clear benefits.
The preconditioned curves stay tightly centered around zero in the left panel, while the unpreconditioned ones display a persistent large negative bias.
At the same time, the right panel indicates that preconditioning reduces the standard deviation by orders of magnitude across all length-scales.

%Second, regarding the choice of distribution, the left panel shows that the mean error is relatively stable for each method across $l$. However, the right panel shows a much stronger sensitivity in variance without preconditioning. The three NP-TSS-$n$ curves are separate, whereas the three AFN-TSS-$n$ curves almost coincide. This indicates that preconditioning not only suppresses bias and variance but also substantially increases TSS robustness with respect to the sampling distribution.

{Second, regarding the choice of distribution, the left panel shows that the mean error is relatively stable for each method across $l$. However, the right panel shows a much stronger sensitivity in variance without preconditioning. As discussed in Section 3.3, because the $\Gamma$-optimal distribution minimizes the theoretical variance upper bound rather than the exact empirical variance, it does not uniformly dominate the baseline choices across all length scales. Instead, its role here is primarily diagnostic. Specifically, the wide separation of the three NP-TSS-$n$ curves illustrates how sensitive the unpreconditioned estimator is to the truncation distribution. In contrast, the near coincidence of the three AFN-TSS-$n$ curves demonstrates that preconditioning effectively suppresses this sensitivity, substantially increasing TSS robustness.}

It is worth mentioning that for this application, the dominant cost per Krylov step is matrix-vector multiplication (MatVec). The cost of applying efficient near-linear complexity preconditioners, such as AFN, is typically substantially lower than that of a full MatVec, even when compared against structured methods such as hierarchical-matrix MatVecs \cite{smash,datadriven}. Due to the reduced iteration counts enabled by preconditioning, AFN-TSS-$n$ offers a favorable error-cost trade-off: it reaches high accuracy at a lower expected per-solve budget.

\subsubsection{Influence of reorthogonalization window size}

In Section~\ref{subsec:stable-precond-lanczos}, we demonstrated that full reorthogonalization improves both solution accuracy and eigenvalue estimation. We now examine its effect on the estimation variance in practical settings. In the following experiments, we evaluate the performance of the Lanczos method for inverse quadratic form and log-determinant estimation under different reorthogonalization window size $i_{\text{orth}}$. To isolate the impact of orthogonality loss, no preconditioner is used, and the iteration bounds are set to $i_{\min}=30$ and $i_{\max}=50$. Because the estimator is less accurate in this configuration, we reduce the dataset size to 1,024 points uniformly drawn from $[0,10]^3$. We continue to focus on the “middle-rank” regime, using $l\in[2,5]$ with $f=1.0$ and $\mu=0.01$. The sampling distribution is again chosen as $\mathbb{P}({Q}=j)\propto e^{-0.5j}$, and each experiment is repeated 1,000 times to estimate the standard deviation.

\begin{figure}[htbp]
    \centering
    \includegraphics[width=0.95\linewidth]{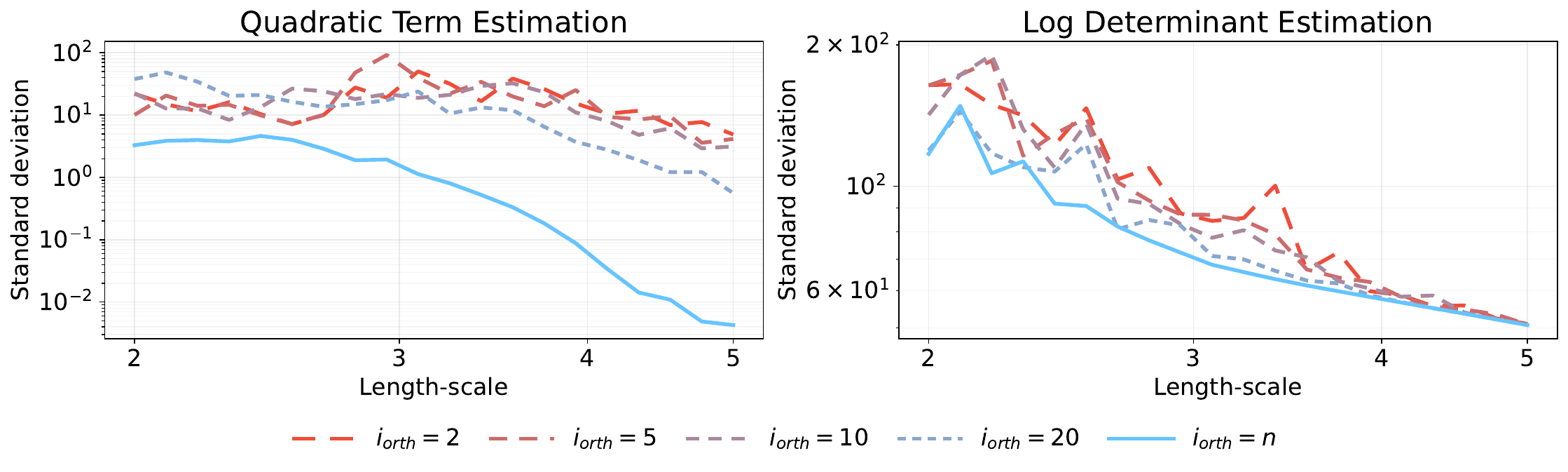}
    \caption{Standard deviation of the inverse quadratic form estimation $\mathbf y^\top \widehat{\mathbf K}^{-1}\mathbf y$ (left) and log-determinant estimation $\log|\widehat{\mathbf K}|$ (right) versus length-scale $l$ (RBF kernel with $f=1.0$, $\mu=10^{-2}$). The unpreconditioned TSS with $\mathbb{P}({Q}=j)\propto e^{-0.5j}$ is tested under different $i_{\text{orth}}$ using $i_{\min}=30$ and $i_{\max}=50$. We sample 1,000 times for each $l$ to report the standard deviation.}
    \label{fig: cg vs lanc variance}
\end{figure}

As shown in Figure~\ref{fig: cg vs lanc variance}, the standard deviation decreases noticeably once the reorthogonalization window becomes sufficiently large, with the most pronounced reduction occurring under full reorthogonalization. This behavior indicates that the loss of orthogonality increases estimator noise, whereas enlarging $i_{\text{orth}}$ stabilizes the Lanczos iterations and reduces variance---albeit at the expense of additional inner products and memory. %This variance-cost trade-off can be further mitigated through preconditioning: with AFN, a much smaller $i_{\max}$ suffices to achieve comparable accuracy, making full reorthogonalization substantially more affordable.

\subsubsection{GP NLML}

In the previous sections, we evaluated the TSS estimator on isolated tasks involving inverse quadratic forms and log-determinants. We now extend the experimental analysis to the full GP NLML and its gradients using the RBF kernel. In this experiment, we employ RBF kernel matrices $\widehat{\mathbf{K}}$ with parameters $f=1.0$, $\mu=0.01$, and $l\in[1.0,10.0]$. Labels are generated as $\mathbf{y}\sim\mathcal{N}(0,\widehat{\mathbf{K}})$, and the loss is evaluated at $(f,l,\mu)=(1.0,l,0.01)$. The dataset consists of 4,096 points uniformly sampled from $[0,16]^3$. We use the sampling distribution $\mathbb{P}({Q}=j)\propto e^{-0.5j}$, set $i_{\min}=5$ and $i_{\max}=15$, and fix both the rank and Schur complement fill level in AFN to 64. We compare AFN-TSS-$n$ against the unpreconditioned NP-TSS-$n$ and the preconditioned AFN-T-$n$ using $i_{\min}$, $\lceil\mathbb{E}({Q})\rceil$, and $i_{\max}$ iterations. Because this experiment is more computationally intensive, each length-scale is averaged over 100 realizations, with a single probe vector used in SLQ per realization.

\begin{figure}[htbp]
    \centering
    \includegraphics[width=0.95\linewidth]{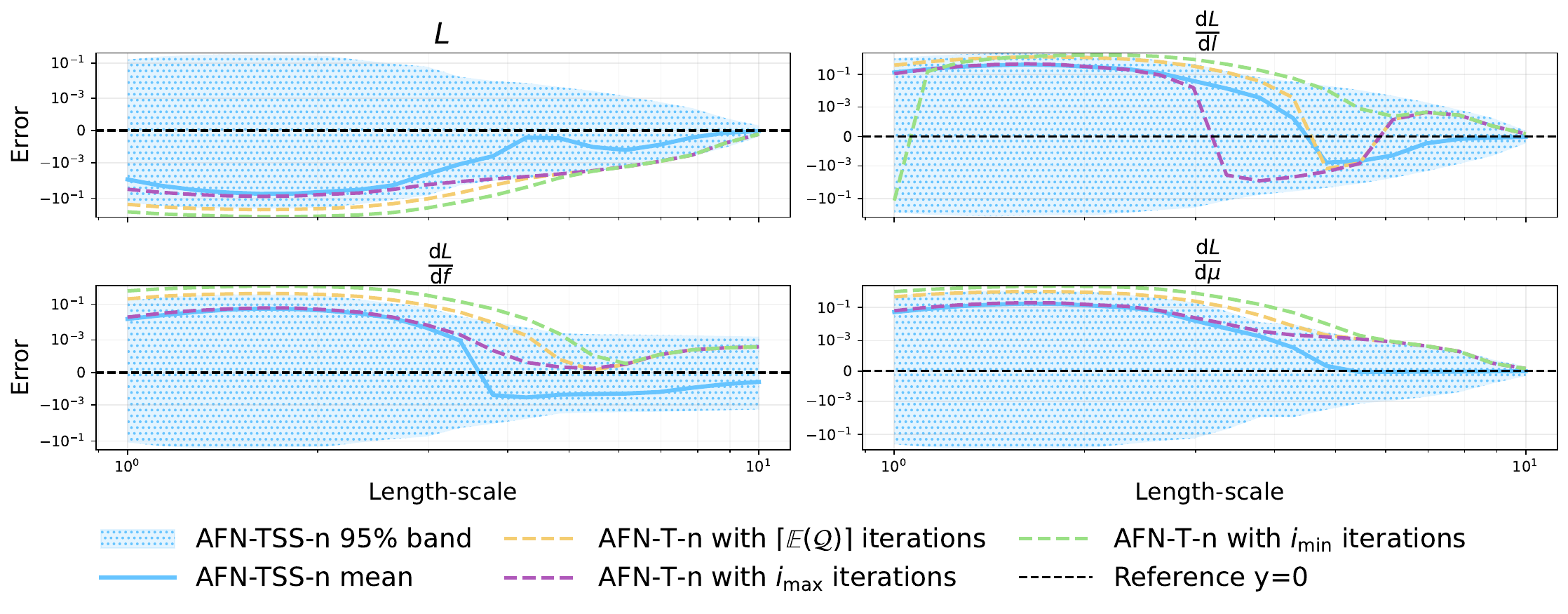}
    \caption{Signed mean error of the GP negative log marginal likelihood and its gradients versus length-scale $l$ (RBF kernel with $f=1.0$, $\mu=10^{-2}$). We use $\mathbb{P}({Q}=j)\propto e^{-0.5j}$ with $i_{\min}=5$ and $i_{\max}=15$ for TSS estimator, fix both the rank and Schur complement fill level for AFN to 64, and sample 100 times for each $l$ to report the mean and the 95\% band. The AFN-TSS-$n$ is compared with AFN-T-$n$ with $i_{\min}$, $i_{\max}$, and $\lceil\mathbb{E}({Q})\rceil=8$ iterations.}
    \label{fig:losses}
\end{figure}

\begin{figure}[htbp]
    \centering
    \includegraphics[width=0.95\linewidth]{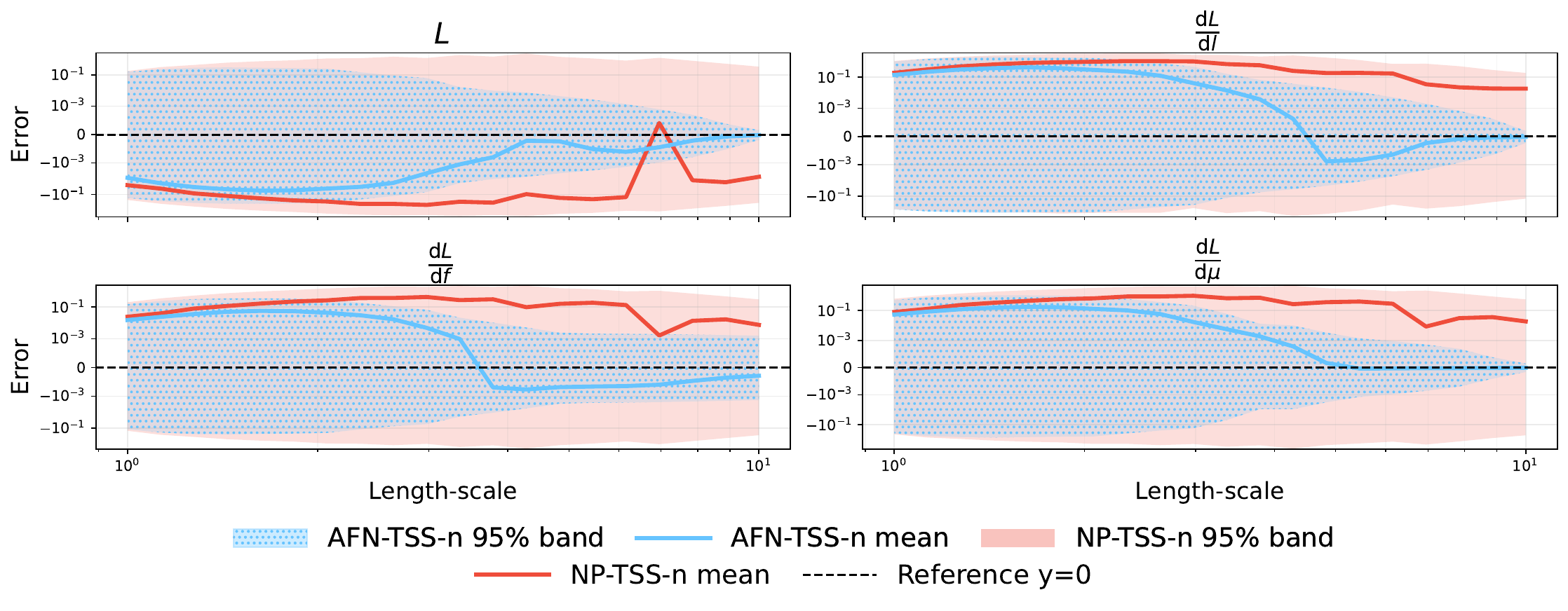}
    \caption{Signed mean error of the GP negative log marginal likelihood and its gradients versus length-scale $l$ (RBF kernel with $f=1.0$, $\mu=10^{-2}$). We use $\mathbb{P}({Q}=j)\propto e^{-0.5j}$ with $i_{\min}=5$ and $i_{\max}=15$, fix both the rank and Schur complement fill level for AFN to 64, and sample 100 times for each $l$ to report the mean and the 95\% band. The AFN-TSS-$n$ is compared with NP-TSS-$n$.}
    \label{fig:losses tss}
\end{figure}

Figure~\ref{fig:losses} and Figure~\ref{fig:losses tss} present results for the GP NLML loss $L$ and its gradients with respect to ${f,l,\mu}$. All quantities are normalized by $n$ to eliminate the trivial linear scaling with dataset size. The trends observed here are consistent with those from the earlier component tests. From Figure~\ref{fig:losses} we see that under a matched iteration budget ($\lceil \mathbb{E}({Q})\rceil$), truncated estimators exhibit a noticeably larger bias in both $L$ and its gradients across most 
length-scales, whereas the TSS estimator achieves mean errors comparable to the more expensive truncated baseline at $i_{\max}$. Figure~\ref{fig:losses tss} further shows that AFN preconditioning suppresses mean errors and significantly tightens the uncertainty bands, particularly for moderate to large length-scales.

\subsection{GP Training}
In the next set of experiments, we evaluate the performance of TSS within the GP hyperparameter-learning loop. 
To visualize the optimization trajectory against the exact NLML contours---a task computationally prohibitive for large-scale datasets---we employ a small dataset of $n=200$ points uniformly sampled from $[0,\sqrt[3]{200}]^3$, allowing the exact NLML contours to be visualized. 
In this tractable regime, we utilize the unpreconditioned NP-TSS-$n$ as a proxy for the preconditioned estimator on large-scale tasks, since preconditioning on large-scale benchmarks is expected to yield spectral properties similar to this naturally better-conditioned small-scale benchmark.
The labels are generated from a Gaussian random field with the RBF kernel and parameters $(f,l,\mu)=(1.0,2.0,0.5)$. Optimization is performed in the unconstrained variables $(\tilde l,\tilde\mu)$, which are mapped to the positive domain via the $\mathrm{softplus}$ transformation, i.e., $(l,\mu)=\mathrm{softplus}(\tilde l,\tilde\mu)$ with $f$ fixed at $1.0$. The initial parameters are set to $(1.0,1.0)$. For TSS, we use the sampling distribution $p(Q=j)\propto e^{-0.5j}$ with $i_{\min}=5$ and $i_{\max}=15$, and employ a single probe vector for both TSS and SLQ. We compare NP-TSS-$n$ against the exact GP and the truncated NP-T-$n$ using $i_{\min}$, $\lceil \mathbb E(Q)\rceil$, and $i_{\max}$ iterations. Results are reported over 3,000 gradient descent iterations with a learning rate of $0.1$.

\begin{comment} 
\begin{figure}[htbp]
    \centering
    \includegraphics[width=0.95\linewidth]{FIGS/new/test06_trajectories.pdf}
    \caption{Optimization trajectory in $(l,\mu)$ space for GP hyperparameters with Gaussian Random Field generated with kernel matrix using $(f,l,\mu)=(1.0,2.0,0.5)$ and $200$ points generated in $[0,\sqrt[3]{200}]^3$. We use $\mathbb{P}({Q}=j)\propto e^{-0.5j}$ with $i_{\min}=5$ and $i_{\max}=15$ for TSS estimator, and use a single sample for TSS and SLQ. The NP-TSS-$n$ is compared with exact GP and NP-T-$n$ with $i_{\min}$, $i_{\max}$, and $\lceil\mathbb{E}({Q})\rceil=8$ iterations. We test RBF kernel (left) and Mat\'ern~3/2 kernel (right). Insets zoom into the final locations.}
    \label{fig:training}
\end{figure}
\end{comment}
\begin{figure}[htbp]
    \centering
    \includegraphics[width=0.95\linewidth]{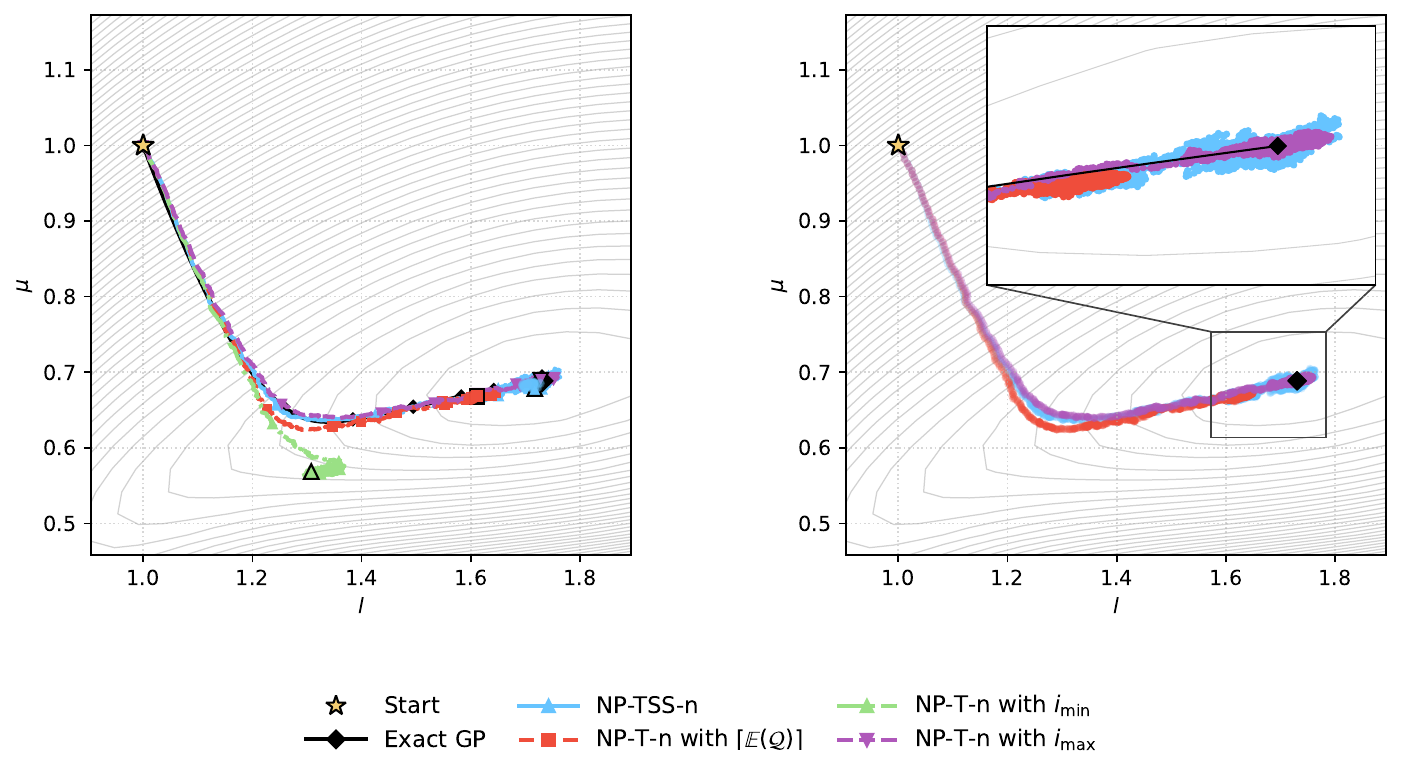}
    \caption{Optimization trajectory in $(l,\mu)$ space for GP hyperparameters with Gaussian Random Field generated with RBF kernel matrix using $(f,l,\mu)=(1.0,2.0,0.5)$ and $200$ points generated in $[0,\sqrt[3]{200}]^3$. We use $\mathbb{P}({Q}=j)\propto e^{-0.5j}$ with $i_{\min}=5$ and $i_{\max}=15$ for TSS estimator, and use a single sample for TSS and SLQ. The NP-TSS-$n$ is compared with exact GP and NP-T-$n$ with $i_{\min}$, $i_{\max}$, and $\lceil\mathbb{E}({Q})\rceil=8$ iterations. We plot the overall trajectory (left) and highlight the local oscillation behavior (right).}
    \label{fig:training}
\end{figure}

As shown in Figure~\ref{fig:training}, the 3,000 gradient descent iterations are sufficient to ensure that all methods reached a stagnation phase under this fixed learning rate.
While the truncated NP-T-$n$ with $i_{\min}$ stops far from the target, increasing the budget to $\lceil\mathbb{E}({Q})\rceil$ improves the path and lands much closer to the optimum.
However, as the zoom inset in the right panel reveals, the method eventually stalls at a sub-optimal location. 
The trajectory cloud demonstrates that the estimator has settled into a stationary distribution distinct from the true minimum, confirming that the remaining error is due to intrinsic bias rather than insufficient iterations.
Using truncated NP-T-$n$ with $i_{\max}$ yields a trajectory that is closer to the optimum, but at a higher cost. 
In contrast, with an expected cost comparable to the $\lceil\mathbb{E}({Q})\rceil$ baseline, NP-TSS-$n$ tracks the exact GP trajectory closely and its final iterate falls in the same basin, clearly demonstrating the advantage of TSS over simple truncation. 
Because we use a single probe, the TSS path shows mild jitter near the end, but the overall direction and endpoint alignment are preserved.

In our last experiment, we evaluate end-to-end optimization over all three hyperparameters on both a real dataset (the Bike dataset from the UCI repository \cite{asuncion2007uci}) and a widely used synthetic benchmark (the 2D Franke function with points sampled uniformly in $[0,1]^2$ \cite{franke1979critical}). In both tests, we set the dataset size to $n=4,096$. For the Bike dataset, we apply per-feature $z$-score normalization and subsample 4,096 points, and use RBF kernel for GP. For the Franke dataset, we generate labels using the \texttt{franke} function in \texttt{MATLAB} and add Gaussian noise $\mathcal N(0,\,0.1^2)$, and use Mat\'ern~3/2 kernel for GP. We optimize in unconstrained variables $(\tilde f,\tilde l,\tilde\mu)$ with $\mathrm{softplus}$ mapping, and initialized $(\tilde f,\tilde l,\tilde\mu)$ at $(0,0,0)$ and run 1,000 full-batch Adam iterations with learning rate $0.01$. For TSS we use $p(Q=j)\propto e^{-0.5j}$ with $i_{\min}=5$ and $i_{\max}=15$ and a single SLQ probe per step; for AFN we set the rank and Schur complement fill level to $64$. We compare NP-TSS-$n$ and AFN-TSS-$n$ against the exact GP, and plot the resulting trajectories in $(f,l,\mu)$.

\begin{figure}[htbp]
    \centering
    \includegraphics[width=0.95\linewidth]{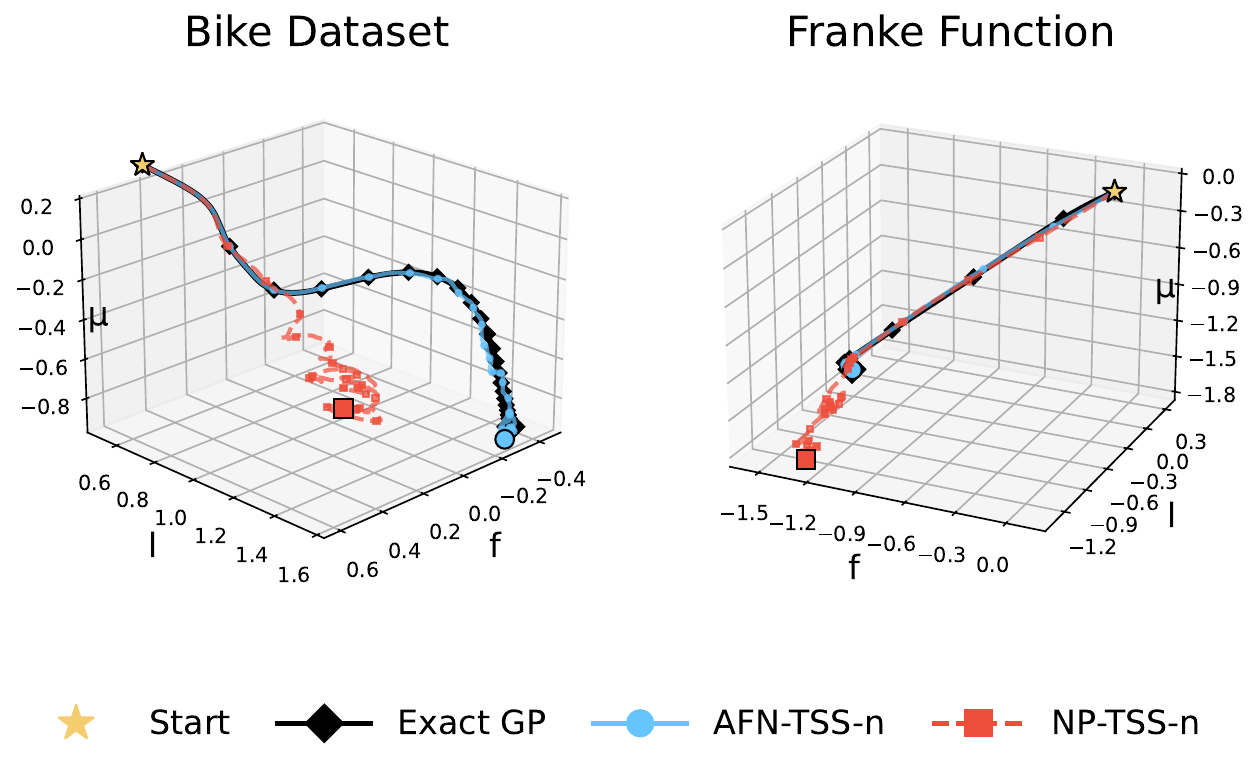}
    \caption{Optimization trajectories in $(f,l,\mu)$ for GP hyperparameters using dataset with $n=4,096$ samples. We use $\mathbb{P}({Q}=j)\propto e^{-0.5j}$ with $i_{\min}=5$ and $i_{\max}=15$ for TSS estimator, and set the rank as well as Schur complement fill level for AFN to 64. \textit{Left:} Subsampled bike dataset with RBF kernel; \textit{Right:} 2D Franke function using data generated in $[0,1]^2$ with Mat\'ern~3/2 kernel.}
    \label{fig:training complete}
\end{figure}

Figure~\ref{fig:training complete} plots the optimization trajectories in $(f,l,\mu)$
for the Bike and Franke cases. In both tests, AFN-TSS-$n$
closely follows the exact GP path, and its final iterate lands in the same basin.
Without preconditioning, NP-TSS-$n$ produces noisier trajectories and often
gets stuck in a nearby but different region. This shows the benefit of the
preconditioned TSS estimator under a matched iteration budget.

\section{Conclusion}\label{sec: conclusion}

We introduced the {Preconditioned Truncated Single-Sample} (PTSS) estimators, a class of stochastic Krylov methods that combine preconditioning with randomized truncation to achieve low-variance and numerically stable estimations of linear solves, log-determinants, and their derivatives. Theoretical analysis established explicit results on the mean, variance, and concentration of PTSS, revealing how preconditioning reduces variance and stabilizes Krylov recurrences in finite precision. Numerical experiments confirmed that PTSS achieves superior accuracy, stability, and computational efficiency compared with existing unbiased and biased alternatives. Together, these results demonstrate that preconditioning and stochastic truncation can be synergistically combined to enable scalable and reliable estimators for large-scale stochastic optimization and Bayesian learning. {The practical benefit of PTSS depends on the preconditioning regime. When preconditioning is near-perfect, a small fixed number of Krylov iterations can already achieve near machine precision, and stochastic truncation offers limited advantage. When preconditioning is weak, convergence is slow and orthogonality/spectral issues can dominate, making variance and stability challenging regardless of truncation strategy. PTSS is designed for the common intermediate ``good but not great'' regime: preconditioning compresses the spectrum enough that a modest truncation depth $i_{\max}$ provides acceptable accuracy, yet running to tight tolerances at every outer iteration remains too expensive and loss of orthogonality can degrade quadrature-based estimates.} Future work will extend the framework to block variants and explore adaptive strategies for dynamic truncation and preconditioner selection.

\appendix

\section{CG for Eigenvalue Estimation}\label{sec: cg eig est}

The basic CG algorithm, as shown in Algorithm~\ref{alg:CG}, can be used to estimate the tridiagonal matrix $\mathbf{T}_m$ of its underlying Lanczos algorithm after $m$-steps. 

\begin{algorithm}[htbp]
\caption{Conjugate Gradient\label{alg:CG}}
\begin{algorithmic}[1]
\STATE{\bf input:} ${\mathbf{A}},\ \mathbf{y}, \ \mathbf{x}_0, \ m$
\STATE{\bf output:} approximate solution $\mathbf{x}_m$
\STATE$\triangleright$ Compute residual $\mathbf{r}_0=\mathbf{y}-{\mathbf{A}}\mathbf{x}_0$.
\STATE$\triangleright$ Set $\mathbf{p}_0=\mathbf{r}_0$
\FOR{$j=0,1,\ldots,m-1$}
\STATE$\triangleright$ Compute $\alpha_j=\frac{\mathbf{r}_j^\top\mathbf{r}_j}{\mathbf{p}_j^\top{\mathbf{A}}\mathbf{p}_j}$.
\STATE$\triangleright$ Compute $\mathbf{x}_{j+1}=\mathbf{x}_j+\alpha_j\mathbf{p}_j$.
\STATE$\triangleright$ Compute $\mathbf{r}_{j+1}=\mathbf{r}_j-\alpha_j{\mathbf{A}}\mathbf{p}_j$.
\STATE$\triangleright$ Compute $\beta_j=\frac{\mathbf{r}_{j+1}^\top\mathbf{r}_{j+1}}{\mathbf{r}_j^\top\mathbf{r}_j}$.
\STATE$\triangleright$ Compute $\mathbf{p}_{j+1}=\mathbf{r}_{j+1}+\beta_j\mathbf{p}_j$.
\ENDFOR
\STATE$\triangleright$ Return $\mathbf{x}_m$.
\end{algorithmic}
\end{algorithm}

If we store all the $\alpha$s and $\beta$s generated from each step of CG, it is known that $\mathbf{T}_m$ is given by:
\begin{equation}
    \mathbf{T}_m=
    \begin{pmatrix}
        \frac{1}{\alpha_0} & \frac{\sqrt{\beta_0}}{\alpha_0} & & & \\
        \frac{\sqrt{\beta_0}}{\alpha_0} & \frac{1}{\alpha_1}+\frac{\beta_0}{\alpha_0} & \frac{\sqrt{\beta_1}}{\alpha_1} & & \\
         & \ddots & \ddots & \ddots & \\
         & &\frac{\sqrt{\beta_{m-3}}}{\alpha_{m-3}} & \frac{1}{\alpha_{m-2}}+\frac{\beta_{m-3}}{\alpha_{m-3}}& \frac{\sqrt{\beta_{m-2}}}{\alpha_{m-2}} \\
         & & & \frac{\sqrt{\beta_{m-2}}}{\alpha_{m-2}} & \frac{1}{\alpha_{m-1}}+\frac{\beta_{m-2}}{\alpha_{m-2}} \\
    \end{pmatrix}.
\end{equation}

\section*{Acknowledgments}

The authors are very grateful to the two anonymous referees for their valuable suggestions.

\bibliographystyle{siamplain}
\bibliography{papers}

\end{document}